\documentclass{siamart171218}
\usepackage{multirow}
\usepackage{braket,amsfonts,amsopn,bm}
\usepackage{graphicx,epstopdf}
\usepackage{cite}
\usepackage{array}
\usepackage[linesnumbered,ruled,vlined,algo2e]{algorithm2e}
\usepackage{mathrsfs}

\usepackage{amssymb,amsfonts,amsmath,latexsym,epsfig,euscript,color}
\usepackage[utf8]{inputenc}
\title{The short-term rational Lanczos method and applications \thanks{This version dated
\today.}}
\author{Davide Palitta\thanks{Research Group Computational Methods in Systems and Control Theory (CSC), 
	Max Planck Institute for Dynamics of Complex Technical Systems, 
	Sandtorstr.\ 1, 39106 Magdeburg, Germany, 
	({\tt palitta@mpi-magdeburg.mpg.de})} 
	\and Stefano Pozza\thanks{Charles University, Faculty of Mathematics and Physics, 
Sokolovsk\'a 83, 186 75 Praha 8, Czech Republic; associated member of ISTI-CNR, Pisa, Italy ({\tt pozza@karlin.mff.cuni.cz})}	
	\and Valeria Simoncini\thanks{Dipartimento 
di Matematica and AM$^2$,
Alma Mater Studiorum - Universit\`a di Bologna, Piazza di Porta S. Donato, 5, I-40127 Bologna, Italy;
and also IMATI-CNR, Pavia, Italy ({\tt valeria.simoncini@unibo.it}).}}

\newcommand{\RR}{\mathbb R}
\newcommand{\CC}{\mathbb C}

\newtheorem{prop}[theorem]{Proposition}

\newtheorem{example}[theorem]{Example}
\newtheorem{remark}[theorem]{Remark}

\newcounter{mymac@matlab}
\newcommand{\MATLAB}{{\sc matlab}% 
  \ifnum\value{mymac@matlab}<1%
  \textsuperscript{\textregistered}%
  \setcounter{mymac@matlab}{1}%
  \fi%
}

\def\b{\boldsymbol}
\DeclareMathOperator*{\argmin}{arg\,min}

\begin{document}
\maketitle

\begin{abstract}
Rational Krylov subspaces have become a reference tool in dimension reduction procedures 
for several application problems.
When data matrices are symmetric, a short-term recurrence can be used to generate an associated orthonormal
basis. In the past this procedure was abandoned because it requires  
twice the number of linear system solves per iteration {\color{black}compared}
 with the classical long-term method.
We propose an implementation that allows one to obtain {\color{black} the} rational subspace
{\color{black} reduced} matrices % without explicitly storing the whole orthonormal basis, 
{\color{black} at lower overall computational costs than proposed in the literature by
also conveniently combining the two system solves.}
%, with a moderate
%computational overhead associated with sparse system
%solves.
%{\color{black} Even though twice the amount of linear system solutions is still needed in
%the basis construction, our novel, block-oriented implementation shows to be competitive with respect to state-of-the-art long-term schemes for rational Krylov subspace computation.}
{\color{black}Several applications are discussed 
where the short-term recurrence feature can be exploited to avoid storing the whole
orthonormal basis. We illustrate the advantages of the proposed procedure with
several examples.}
\end{abstract}

%%%%%%%%%%%%%%%%%%%%%%%%%%%%%%%%%%%%%%%%
\section{Introduction}
Given a symmetric matrix $A\in{\mathbb R}^{n\times n}$ and a unit norm vector $v\in{\mathbb R}^{n}$,
we are interested in analyzing the algebraic recurrence that generates the rational Krylov subspace
\begin{equation}\label{eqn:RKS}
{\cal K}_m(A,v,\b{\xi}_m) =
	{\rm span}\left\{v, (I - \xi_1^{-1} A )^{-1} v, 
	\ldots, \prod_{j=1}^{m-1} (I - \xi_j^{-1} A )^{-1} v\right\},
\end{equation}
and in the applicability of the computed quantities; here 
$\b{\xi}_m=[\xi_1, \ldots, \xi_{m-1}]$ with $\xi_i\ne 0$ are 
such that $I - \xi_i^{-1} A$ is nonsingular for $i=1, \ldots, m-1$.
%
% This is ensured
%if the eigenvalues of $A$ and $\xi_i$ have opposite sign, an hypothesis that
%will be assumed throughout. 
%
In the considered applications 
$A$ is definite (either positive or negative), for which strong theoretical 
arguments for the selection of the $\xi_i$ have been discussed, 
see~{\color{black}\cite{Druskin2010,Guettel2013}} and references therein.
{\color{black}In our context, the corresponding theoretical setting suggests 
that the eigenvalues of $A$ and $\xi_i$ have opposite sign, an hypothesis that
we will assume throughout.}
By relying on efficient sparse solvers for
linear systems, rational Krylov subspaces have become a major tool in a variety
of application problems, including eigenvalue approximation, 
dynamical system reduction, matrix equation solution,
and matrix function and bilinear form evaluations 
\cite{Gugercin2008,BecRei09,DruKniZas09,DruKniSim11,GalGriVDo96,GuettelPhD.10,Ruhe1984}.
An orthonormal basis $\{q_1, \ldots, q_m\}$
for ${\cal K}_m(A,v,\b{\xi}_m)$ can be determined by using the Gram-Schmidt procedure as $q_1=v$ and
$$
\hat q_{j+1} := \left(I - \xi_j^{-1} A\right )^{-1}  q_j, \qquad
h_{j+1,j} q_{j+1} = \hat q_{j+1} - Q_j h_j,
$$
where $h_j = Q_j^T \hat q_{j+1}$, $Q_j=[q_1, \ldots, q_j]$
and $h_{j+1,j}$ is the normalization factor for $q_{j+1}$ for $j=1, \ldots, m$ \cite[Eq.~(4.15)]{Ruhe1984}.
This can be viewed as a rational variant of the Arnoldi iteration.
The parameters ({\color{black}or shifts}) $\b{\xi}_m=[\xi_1, \ldots, \xi_{m-1}]$ can be computed a-priori or determined adaptively
as the space grows. The vector $h_j$ contains the orthogonalization coefficients, so that $Q_j$
has orthonormal columns for $j=1, \ldots, m$.

For $A$ symmetric,
in \cite{Deckers.Bultheel.07} a short-term recurrence was introduced to generate
an orthonormal basis of rational functions associated with ${\cal K}_m(A,v,\b{\xi}_m)$, from which
a short-term %{\em vector} 
recurrence can be derived for the basis $\{q_1, \ldots, q_m\}$ \cite{GuettelPhD.10}.

{\color{black}
The short-term recurrence yielding an orthonormal basis of 
the {\em polynomial} Krylov subspace generated by a symmetric $A$}
 was introduced by Lanczos \cite{Lanczos1952}, {\color{black}and it is}
 thus referred to as the Lanczos iteration.
This recurrence allows one to store few {\color{black}$n$-dimensional} vectors, leading to major savings in various
approximation problems where the whole basis is otherwise not needed.
Thanks to the work in
\cite{Deckers.Bultheel.07,GuettelPhD.10}, similar advantages can be envisioned 
in the rational case. 
{\color{black} 
G\"uttel in \cite[p.~46]{GuettelPhD.10} says:
{\it ``Note that in each iteration of Algorithm 2 two linear systems with $I - A/\xi_j$
need to be solved [...]. Hence, this algorithm is
in general not competitive with the rational Arnoldi algorithm if the poles $\xi_j$ vary often.
Moreover, we will make explicit use of the orthogonality of the rational Krylov basis $V_{m+1}$
when computing Rayleigh–Ritz approximations for $f(A)b$ (see Chapter 6). In this case
full orthogonalization of $V_{m+1}$ is required anyway and one cannot take advantage of the
short recurrence''.}
In spite of the elegant derivation, these considerations led G\"uttel
 to discard the {\color{black} short-term iteration} in his application
setting.}  % because deemed to be too expensive, 
%especially when compared to long-term schemes which necessitate a single linear system solution. 
%Moreover, the application 
%in \cite{GuettelPhD.10} 
% required the whole basis to compute the sought final approximation, 
%hence the use of the short-term recurrence did not provide any 
%practical benefit even in terms of storage demand.}

We claim that in spite of the extra cost per iteration, the rational short-term recurrence 
{\color{black} does} provide an appealing framework for a variety of approximation problems.
Our contribution is two-fold. First, we propose a new implementation of the short-term recurrence that:
i) alleviates the computational costs associated with the two solves by combining them into a single
linear system solve with multiple right-hand sides; ii) derives the entries of the
%key
reduced matrix $J_m :=Q_m^T A Q_m$ as the iterations proceed, without 
{\color{black} using $Q_m$ or explicitly solving $m\times m$ linear systems}.
%storing the whole basis matrix $Q_m=[q_1, \ldots, q_m]$.
Second, we illustrate the advantages of the obtained implementation in the numerical treatment
of several application problems
{\color{black}that do not require the whole basis matrix $Q_m=[q_1, \ldots, q_m]$;}
% that are typically solved by the rational Krylov method; 
these include the approximation of quadratic and bilinear
forms in general, and quantities of interest in control, such as
the estimation of the $\mathcal{H}_2$-norm of (parametric) linear time-invariant systems and of the optimal feedback control 
function.%\footnote{To illustrate the applications of interest in a self-contained manner, we 
%provide a short introductory description for each of them. % of the state-of-the art.  
%As a consequence, the bibliography necessarily includes some of the relevant related results. 
%To keep the reference list at an acceptable length we have constrained the number of references,
%which is however still significantly large.}.
{\color{black} We will refer to this implementation as the $Q_m$-less computation.}

We also start a discussion on the behavior of the obtained recurrence in finite precision arithmetic.
Our matrix relations and experimental evidence seem to suggest that the short-term iteration is affected
by round-off error accumulations similar to those of the classical Lanczos method. A deeper analysis of this
crucial aspect deserves a dedicated research, it will thus be postponed to future work.

The synopsis of the paper is as follows.  In section~\ref{The short-term rational Krylov iteration} 
we revisit the rational Lanczos iteration. An efficient, basis-free procedure 
to compute the matrix $J_m=Q_m^TAQ_m$ is derived in section~\ref{On the computation of J},
while in section~\ref{The overall procedure} the novel implementation of the rational Lanczos method 
is illustrated. 
A panel of applications where the rational Lanczos method can be successfully employed 
is presented in section~\ref{Applications}. These include numerical approximations of 
quadratic and bilinear forms (section~\ref{sec:bilinear}), matrix function trace 
estimation (section~\ref{sec:trace}), $\mathcal{H}_2$-norm computation for 
LTI systems (section~\ref{H2 norm computation}), 
and LQR feedback control approximations (section~\ref{LQR feedback control}). 
In section~\ref{sec:fpa} some preliminary remarks on the behavior of the rational 
Lanczos method in finite precision arithmetic are reported. 
Our conclusions are given in section~\ref{Conclusion}, while in the 
appendix  the {\em block} rational Lanczos algorithm is given.

%%%%%%%%%%%%%%%%%%%%%%%%%%%%%%%%%%%%%%%%
\section{The short-term rational Krylov iteration}\label{The short-term rational Krylov iteration}
Based on the thorough analysis of orthogonal rational functions in \cite{Bultheeletal03},
the authors of \cite[Th.~4.2]{Deckers.Bultheel.07} developed a three-term recurrence 
relation to generate a sequence of
orthogonal rational functions associated with the rational Krylov subspace (\ref{eqn:RKS}); see also \cite{Deckers.Bultheel.07}.
This elegant construction was further {\color{black} developed} in
\cite[Section 5.2]{GuettelPhD.10}, leading to the following vector recurrence for $j\geq1$,
\begin{eqnarray}\label{eq:three_term}
	\beta_j\left(I-\xi_j^{-1} A\right) q_{j+1} 
	= A q_j - \alpha_j \left (I- \xi_{j-1}^{-1} A\right ) q_j 
	- \beta_{j-1} \left (I- \xi_{j-2}^{-1} A\right ) q_{j-1},
\end{eqnarray}
where, with the usual convention that $1/\infty=0$, $\xi_{-1}:=\infty$, 
$\xi_{0}:=\infty$, $\beta_0=0$, $q_0=0$, and $q_1=v/\|v\|$. By setting 
$$
r := (I-\xi_j^{-1} A)^{-1}( Aq_j +
\beta_{j-1} \xi_{j-2}^{-1} Aq_{j-1}) - \beta_{j-1}q_{j-1}, \,\quad
s :=(I-\xi_j^{-1} A)^{-1}(I-\xi_{j-1}^{-1} A) q_j,
$$
 the coefficients $\alpha_j$ and $\beta_j$ 
are computed as $\alpha_j = (r^T q_j)/(s^Tq_j)$, and $\beta_j = \|r - \alpha_j s\|$. A 
simple rearrangement of the terms leads to the following more compact notation,
$$
A [ q_{j-1}, q_j, q_{j+1} ] 
\begin{bmatrix} 
\frac{\beta_{j-1}}{\xi_{j-2}} \\ \frac{\alpha_j}{\xi_{j-1}} + 1 \\ \frac{\beta_j}{\xi_j}
 \end{bmatrix}
=
[ q_{j-1}, q_j, q_{j+1} ] \begin{bmatrix} \beta_{j-1} \\ \alpha_j \\ \beta_j \end{bmatrix},
$$
which, after $m$ iterations gives the Arnoldi-like relation
\begin{equation}\label{eqn:Q_arnoldi}
A Q_{m+1} \underline{K}_m = Q_{m+1} \underline{H}_m, \quad {\rm with}\quad
\underline{K}_m = \underline{I}_m + D_m^{-1} \underline{H}_m\in{\mathbb R}^{(m+1)\times m} ,
\end{equation}
where $\underline{I}_m$ is the $(m+1)\times m$ identity matrix and 
$D_m={\rm diag}(\xi_0, \ldots, \xi_{m})$. Hence %, with $\xi_0=\infty$.  Hence,
{\footnotesize
$$
%\underline{K}_j =
%\begin{bmatrix}
%1 & 0 &    &   &   \\
%\frac{\beta_1}{\xi_1} & \frac{\alpha_2}{\xi_1} +1 & \frac{\beta_2}{\xi_1} &   & \\
%  & \frac{\beta_2}{\xi_2}  & \frac{\alpha_3}{\xi_2} +1 & \frac{\beta_3}{\xi_2} &   \\
%  &   & \frac{\beta_3}{\xi_3} & \frac{\alpha_4}{\xi_3} + 1 &\ddots \\ 
%  &   & \ddots        & \ddots  & \ddots
%\end{bmatrix} \in \RR^{(j+1)\times j} , 
\underline{K}_m =
\begin{bmatrix}
1 & 0 &    &    \vspace{3pt}  \\
\frac{\beta_1}{\xi_1} & 1\hspace{-2pt}+\hspace{-2pt}\frac{\alpha_2}{\xi_1} & \frac{\beta_2}{\xi_1}    \\
  & \frac{\beta_2}{\xi_2}  &1\hspace{-2pt}+\hspace{-2pt}\frac{\alpha_3}{\xi_2}  &  \hspace{+2pt}\ddots & \\
	& & \hspace{-2pt}\ddots  & \ddots  & \frac{\beta_{m-1}}{\xi_{m-2}} \vspace{4pt} \\
    & &       & \frac{\beta_{m-1}}{\xi_{m-1}}  & 1\hspace{-2pt}+\hspace{-2pt}\frac{\alpha_m}{\xi_{m-1}} \vspace{4pt} \\
	& & & & \frac{\beta_m}{\xi_m} \vspace{4pt}
\end{bmatrix}, %\in \RR^{(j+1)\times j} ,
\quad
\underline{H}_m =
\begin{bmatrix}
\alpha_1 \vspace{3pt} & \beta_1 &   &    \\
\beta_1  & \alpha_2& \beta_2 &    \\
& \beta_2 & \alpha_3 &  \ddots  \\
 & & \ddots & \ddots & \beta_{m-1} \vspace{4pt} \\
%         & \beta_2 & \alpha_3 & \beta_3& \\
%& & \ddots & \ddots & \ddots\\
& & & \beta_{m-1} & \alpha_m \vspace{4pt} \\
& & & & \beta_m \vspace{4pt}
\end{bmatrix} .% \in \RR^{(j+1)\times j}.
$$
}
The relation in (\ref{eqn:Q_arnoldi}) can also be written as
\begin{equation}\label{eq:rat:lanc:mtx}
	A Q_{m} K_m = Q_{m} H_m + \beta_m \left(I- \xi_m^{-1} A\right ) q_{m+1} e_m^T ,
\end{equation}
with $K_m$ and $H_m$ the leading  $m\times m$ upper parts of $\underline{K}_m$ and $\underline{H}_m$,
respectively. Here and in the following, $e_j$ denotes the $j$th column of the identity matrix,
whose dimension is clear from the context.
Hence, we resume the standard form for the rational Krylov iteration \cite{Ruhe1984}. 
{\color{black} This means that in exact arithmetic
the rational Lanczos and rational Arnoldi recurrences
compute the same basis, provided the same set of shifts is employed in the basis construction.}

{\color{black}Thanks} to the irreducibility of $\underline{H}_m$, {\color{black} it follows from  (\ref{eq:rat:lanc:mtx})} 
that {\color{black}if}
 the matrix $K_m$ is singular, {\color{black} then} $[Q_m, (I-\xi_m^{-1} A)q_{m+1}]$ is not full rank, 
which occurs if and only if the
subspace $\mathcal{K}_m(A,v, \b{\xi}_m)$ is $A$-invariant, i.e., 
$A\mathcal{K}_m(A,v, \b{\xi}_m) = \mathcal{K}_m(A,v, \b{\xi}_m)$. In the $A$-invariant case we have a 
lucky termination of the algorithm.
% see, e.g., \cite{GutKni13} for a similar derivation using a different recurrence. 
Hereafter %, if not specified otherwise, 
we thus always assume $K_m$ to be invertible.
Like for the standard Lanczos procedure, the matrix $J_m=Q_m^TAQ_m$ represents
the projection and restriction of $A$ onto the range space of $Q_m$. However, while
in the (standard) Arnoldi process this matrix is tridiagonal, the matrix $J_m$ is generally full.
A detailed analysis of the structure and decay properties of the entries of $J_m$ 
can be found in \cite{Pozza.Simoncini.20tr}; see also \cite{PraRei14}.
The matrix $J_m$ appears in projection methods for solving problems 
such as linear and quadratic matrix equations and matrix functions evaluations.
%Though the rational Krylov basis can be generated recursively by (\ref{eq:three_term})
%without storing all previous basis vectors, 
{\color{black}Classical implementations of the rational Lanczos recurrence
rely on the whole matrix $Q_m$.
In the following we show that this can be avoided, leading to memory savings in a variety of application problems.}
%{\color{black}Standard implementations of the rational Lanczos recurrence} 
%require the whole matrix $Q_m$ for the computation of $J_m$. 
%In the following we show that this requirement can in fact be
%relaxed, leading to memory savings in a variety of application problems.
%entries of the symmetric matrix
%$J_j$ can be obtained iteratively without storing $Q_j$. 
%This allows us to
%use rational Krylov subspaces with very few memory allocations for approximating
%quantities of interest in a variety of application problems.

%%%%%%%%%%%%%%%%%%%%%%%%%%%%%%%%%%%%%%%%
\subsection{On the computation of $J_j=Q_j^TAQ_j$}\label{On the computation of J}
In this section
 we derive a recurrence for computing {\color{black} the small dimensional matrix $J_j=Q_j^TAQ_j$
without using $Q_j$ or explicitly solving small size linear systems with  $K_j$ at each iteration $j$.}% 
% that avoids storing the whole basis $Q_j$. 

Setting $J_{j+1}=Q_{j+1}^T A Q_{j+1}$ and multiplying 
(\ref{eq:rat:lanc:mtx}) by $Q_{j+1}^T$ from the left, we obtain
\begin{equation}\label{eqn:Jrect}
J_{j+1} \underline{K}_j = \underline{H}_j. 
\end{equation}
%The matrices $\underline{K}_j,  \underline{H}_j$ are rectangular, therefore, at a first glance, it seems we
% cannot get a clean relation for $J_{j+1}$. 
Considering the first $j$ rows we can write
%Nonetheless, we can write
%$$
$J_{j} K_j + w_j e_j^T = {H}_j$, with $w_j = Q_j^T A q_{j+1} \frac {\beta_j}{\xi_j}$,
%J_{j}^{-1}  {H}_j = K_j + J_{j}^{-1} w e_j^T, \quad w = Q_j^T A q_j \beta_j/\xi_j ,
%$$
that is, except for the last column, the matrix $J_{j}  {K}_j$ is tridiagonal, and
\begin{equation}\label{eq:J:exp}
J_{j} = {H}_j K_j^{-1} - w_j e_j^T K_j^{-1}.
\end{equation}
To get a $Q_j$-less computation of $J_j$ we need to obtain a different expression for $w_j$.
%We cannot get a direct computation of $J_j$ unless we obtain an expression for $w_j$
%that does not require storing the whole $Q_j$.
Setting $u = Q_j^T A q_{j+1}$ and exploiting symmetry, 
from (\ref{eqn:Jrect}) we have
\begin{equation}\label{eqn:Jj+1}
\begin{bmatrix} J_j & u \\ u^T & \eta \end{bmatrix}
\begin{bmatrix} I \\ \beta_j/\xi_j e_j^T K_j^{-1} \end{bmatrix} 
=
\underline{H}_j K_j^{-1} ,
\end{equation}
with $w_j= u \beta_j/\xi_j$.
For the last row it holds  that
$u^T = e_{j+1}^T \underline{H}_j K_j^{-1} - \eta \beta_j/\xi_j e_j^T K_j^{-1}$,
which gives the sought-after expression for $w_j$. %= u \beta_j/\xi_j$.
Summarizing, at step $j+1$ we can completely define $J_j$ 
without storing the whole $Q_j$. In particular, its last column is given by 
\begin{align}\label{eq:J_lastcolumn}
J_{j}e_j =& {H}_j K_j^{-1}e_j - w_j e_j^T K_j^{-1}e_j \notag\\
=& {H}_j K_j^{-1}e_j - \left(K_j^{-T}\underline{H}_j^Te_{j+1}   - K_j^{-T}e_j\eta \beta_j/\xi_j  \right)\beta_j/\xi_j e_j^T K_j^{-1}e_j \notag\\
=& {H}_j K_j^{-1}e_j - K_j^{-T}e_j\left(\xi_{j} - \eta \right)\frac{\beta_j^2}{\xi_j^2}  \left( e_j^T K_j^{-1}e_j\right).
%=& {H}_j K_j^{-1}e_j - K_j^{-T}e_j\left(\beta_{j} - \eta \beta_j/\xi_j\right)\beta_j/\xi_j  \left( e_j^T K_j^{-1}e_j\right).
\end{align}
The most expensive steps in \eqref{eq:J_lastcolumn} are the solution of the linear 
systems with $K_j$ and $K_j^T$. The tridiagonal structure of these matrices
allows us to derive a recurrence for the two solution vectors as the iterations proceed,
making the overall computation cheaper than explicitly solving the linear systems from scratch at each iteration $j$.

\begin{lemma}\label{lemma:thomas}
%At the $j$-th iteration of Algorithm~\ref{alg:rational_lanczos}
% the solution $y_j$, $t_j\in\mathbb{R}^j$ to the linear systems
% $K_jy=e_j$ and $K_j^Tt=e_j$,
% %%
% %$K_jy=e_j,\quad\text{and}\quad K_j^Tt=e_j,$
% %%
% can be computed by
 %%
%The iteration in \ref{eq:rat:lanc:mtx} generates a sequence of nested matrices
%$\{K_j\}_{j\ge 1}$, $K_j\in{\mathbb R}^{j\times j}$ and associated 
With the previous notation, for $j=1, \ldots, m$ the solutions $(y_j, t_j)$ to the systems
$K_jy=e_j$ and $K_j^Tt=e_j$ can be obtained via the following recurrences
\begin{equation}\label{eqn:yt}
y_j = \frac 1 {\omega_j} \left (
e_j - \frac{\beta_{j-1}}{\xi_{j-2}} \begin{bmatrix} y_{j-1} \\ 0
       \end{bmatrix}\right ) ,
\quad
t_j = \frac 1 {\omega_j} \left (
e_j - \frac{\beta_{j-1}}{\xi_{j-1}} \begin{bmatrix} t_{j-1} \\ 0
       \end{bmatrix}\right ) , \quad j\geq2, 
\end{equation}
% $$y_j= \begin{bmatrix}
%        -y_{j-1}\frac{\beta_{j-1}}{\xi_{j-2}u_j} \\
%        1/u_j\\
%       \end{bmatrix},\quad\text{and}\quad 
%       t_j=\begin{bmatrix}
%        -t_{j-1}\frac{\beta_{j-1
%        }}{\xi_{j-1}u_j} \\
%        1/u_j\\
%       \end{bmatrix},$$
       %%
%where $y_{j-1}$, $t_{j-1}\in\mathbb{R}^{j-1}$ are the solutions to the same type of systems
%at the previous iteration, that is $K_{j-1}y_{j-1}=e_{j-1}$ and $K^T_{j-1}t_{j-1}=e_{j-1}$,
with $y_1=1$ and $t_1=1$, 
 where $\omega_j\in\mathbb{R}$ is given by the following recursive formula
\begin{equation}\label{eqn:diagU}
\omega_1=1,\;\omega_2=\frac{\alpha_2}{\xi_{1}}+1, 
\quad \omega_j=\frac{\alpha_j}{\xi_{j-1}}+1-
\frac{\beta_{j-1}^2}{\xi_{j-1}\xi_{j-2}\omega_{j-1}},\; j\geq3.
\end{equation}
\end{lemma}

\begin{proof}
We focus on the computation of $y_j=K_j^{-1}e_j$. The computation of $t_j=K_j^{-T}e_j$ is analogous.
Under the assumption that ${\mathcal K}_j(A,v,\b{\xi}_j)$ is not $A$-invariant, $K_i$ is a 
nonsingular matrix, for $i=1,\dots,j$.
%Since $\det(K_1),\dots,\det(K_j) \neq 0$ are all the leading principal minors of $K_j$, 
Hence the LU factorization  $K_j= L_jU_j$ with no pivoting exists, $y_j = U_j^{-1} L_j^{-1} e_j$,
 and the factors are given by
{\footnotesize
$$
L_j=\begin{bmatrix}
       1 & & & & \\
       \ell_2 & 1 & & & \\
             & \ell_3 & \ddots & & \\
             & & \ddots & \ddots & \\
             & & & \ell_j & 1\\
      \end{bmatrix} \in\RR^{j\times j}, \quad 
U_j=\begin{bmatrix}
      \omega_1 & 0 & & & \\
      & \omega_2 & \beta_2/\xi_1 & & \\
      & & \ddots & \ddots & \\
      & & & \ddots & \beta_{j-1}/\xi_{j-2} \\
      & & & & \omega_j\\
      \end{bmatrix} \in\RR^{j\times j}.
$$}
%Hence, $y_j = U_j^{-1} L_j^{-1} e_j$.
Thanks to the structure of $L_j$,  $x=L_j^{-1}e_j=e_j$. To solve
 $U_jy=e_j$ we first determine the diagonal elements of $U_j$.
Direct computation gives the recursion in (\ref{eqn:diagU}), that is
%%
%$$\omega_1=1,\;u_2=\frac{\alpha_2}{\xi_{1}}+1, \quad u_i=\frac{\alpha_i}{\xi_{i-1}}+1-\frac{\beta_{i-1}^2}{\xi_{i-1}\xi_{i-2}u_{i-1}},\; i=3,\ldots,j.
%$$
%%
the computation of $\omega_j$ only requires information available in the current subspace. 
Moreover, if $y_{j-1}\in\mathbb{R}^{j-1}$ is such that $K_{j-1}y_{j-1}=e_{j-1}$, the 
solution $y_j$ to $U_jy=e_j$ can be derived as in~\eqref{eqn:yt}.
%%
%$$ y_j=\begin{bmatrix}
%        -y_{j-1}\frac{\beta_{j-1}}{\xi_{j-2}\omega_j} \\
%        1/\omega_j\\
%       \end{bmatrix}.
%$$
%The same exact reasoning shows that the solution $t_j$ to $K_j^Tt=e_j$ can be written as as in (\ref{eqn:yt}).
%$$ t_j=\begin{bmatrix}
%        -t_{j-1}\frac{\beta_{j-1
%        }}{\xi_{j-1}\omega_j} \\
%        1/\omega_j\\
%       \end{bmatrix},
%$$
Analogously, the solution $t_j$ to $K_j^Tt=e_j$ is obtained as in~(\ref{eqn:yt}),
where $t_{j-1}\in\mathbb{R}^{j-1}$ is such that $K_{j-1}^Tt_{j-1}=e_{j-1}$. 
%We point out that also in the computation of $t_j$ we use the scalar $\omega_j$. Indeed, 
%even though $K_j$ is not symmetric, a careful inspection of the recurrence for 
%the diagonal entries of the upper triangular factors of $K_j$ and $K_j^T$ shows that 
%in both cases these components are equal to the $\omega_j$'s defined above. 
\end{proof}
\vskip 0.1in

\begin{theorem}\label{th:J_jlastcolumn}
With the notation and results of Lemma~\ref{lemma:thomas},
at the $j$-th iteration  %of Algorithm~\ref{alg:rational_lanczos}
 the last column (or row) of $J_j=Q_j^TAQ_j$ is given by 
 $$ 
J_{j}e_j=\widehat y_j -\frac{\beta_j^2}{\xi_j^2} \frac{\xi_j - \eta}{\omega _j} t_j,
$$
where $t_j$ is defined as in Lemma~\ref{lemma:thomas} while $\widehat y_j := H_j y_j$ satisfies
$$\widehat y_1=\alpha_1,\quad \widehat y_j=  \begin{bmatrix}
        -\widehat y_{j-1}\frac{\beta_{j-1}}{\xi_{j-2}\omega_j} \\
        \beta_{j-1}e_{j-1}^Ty_{j} +\frac {\alpha_j}{\omega_j}\\
       \end{bmatrix}+\frac{\beta_{j-1}}{\omega_j}e_{j-1}, \quad \text{for}\;j>1.$$
\end{theorem}
%%
%\begin{proof}
{\it Proof.} By plugging the expressions of $y_j=K_j^{-1}e_j$ and $t_j=K_j^{-T}e_j$ given 
in Lemma~\ref{lemma:thomas} into \eqref{eq:J_lastcolumn} we get
\begin{align*}
 J_{j}e_j =&  {H}_j K_j^{-1}e_j - K_j^{-T}e_j
(\xi_{j} - \eta )\frac{\beta_j^2}{\xi_j^2}  \left( e_j^T K_j^{-1}e_j\right)=H_jy_j
-\frac{\beta_j^2}{\xi_j^2}\frac{\xi_{j} - \eta }{\omega_j}
t_j.
 %=  \widehat y_j
%-\beta_j^2\frac{1 - \eta /\xi_j}{\xi_ju_j}
%t_j
\end{align*}
Using Lemma~\ref{lemma:thomas} and
the tridiagonal structure of $H_j$ we can write $\widehat y_j=H_jy_j$ as 
$$
\widehat y_j= \begin{bmatrix}
        -H_{j-1}y_{j-1}\frac{\beta_{j-1}}{\xi_{j-2}\omega_j} \\
        \beta_{j-1}e_{j-1}^Ty_{j} +\frac{\alpha_j}{\omega_j}\\
       \end{bmatrix}+\frac{\beta_{j-1}}{\omega_j}e_{j-1}= 
        \begin{bmatrix}
        -\widehat y_{j-1}\frac{\beta_{j-1}}{\xi_{j-2}\omega_j} \\
        \beta_{j-1}e_{j-1}^Ty_{j} +\frac{\alpha_j}{\omega_j}\\
       \end{bmatrix}+\frac{\beta_{j-1}}{\omega_j}e_{j-1} .  \qquad \square
$$
%% 
%and the proof is concluded.
% \end{proof}
%%
\vskip 0.1in
 
Theorem~\ref{th:J_jlastcolumn} also
shows that by storing the low dimensional vectors $y_{j-1}$, $\widehat y_{j-1}$, 
and $t_{j-1}$, along with some additional scalar quantities, the allocation of the 
matrices $H_j$ and $K_j$ can be avoided.

\begin{algorithm}[t]
  \DontPrintSemicolon
  \SetKwInOut{Input}{input}\SetKwInOut{Output}{output}
  %%%%%%%%%%% INPUT %%%%%%%%%%%
  \Input{$A\in\mathbb{R}^{n\times n}$, $v\in\mathbb{R}^{n}$, $m\in{\mathbb N}$, $m>0$,
 $\b{\xi}_m=[\xi_1, \ldots, \xi_m]$.}
  %%%%%%%%%%% OUTPUT %%%%%%%%%%%
  \Output{$J_{m}\in\mathbb{R}^{m\times m}$ s.t. $J_{m}=Q_m^TAQ_m$ where $\text{Range}(Q_m)={\cal K}_m(A,v,\b{\xi}_m) $.}
  %%%%%%%%%%%%%%%%%%%%%%%%%%%%%%%%%%% 
  \BlankLine
  Set $\widehat q=v/\|v\|$, $\xi_{-1}:=\infty, \xi_0:=\infty, \beta_0=0, \bar q=0$\;
  
  \For{$j=1, \ldots, m$}{
  %\While{$j\leq m$}{
%  \If{$j=1$}{
%  Set $\widetilde r=A\widehat q$ and $\widetilde s=\widehat q$\;
%  }\ElseIf{$j=2$}{
%  Set $\widetilde r=A\widehat q-\beta_{j-1}\bar q$ and $\widetilde s=(I-A/\xi_{j-1})\widehat q$}
%\Else{
%Set $\widetilde r=A\widehat q-\beta_{j-1}(I-A/\xi_{j-2})\bar q$ and $\widetilde s=(I-A/\xi_{j-1})\widehat q$}
    Set $\widetilde r=A\widehat q-\beta_{j-1}(I-A/\xi_{j-2})\bar q$ and $\widetilde s=(I-A/\xi_{j-1})\widehat q$\;
    
    Solve  $(I-\frac 1 \xi_j A)[r,s]=[\widetilde r,\widetilde s]$\;
    
    Compute $\alpha_j=\frac{r^T\widehat q}{s^T\widehat q}$\;
    
    Set $q=r-\alpha_j s$\;
    
    Set $\beta_j=\|q\|$, $\bar q=\widehat q$, $\widehat q=q/\beta_j$\;
    
    \If{j=1}{
    Set $\omega_1=y_1=t_1=1$ and $\widehat y_1=\alpha_1$
    }\Else{
    Set $\omega_j=\alpha_j/\xi_{j-1}+1-\beta_{j-1}^2/(\xi_{j-1}\xi_{j-2}\omega_{j-1})$ \;
    
    Set
    $y_j=\begin{bmatrix}
        -y_{j-1}\frac{\beta_{j-1}}{\xi_{j-2}\omega_j} \\
        \frac 1 \omega_j\\
       \end{bmatrix}
    $,
    $t_j= \begin{bmatrix}
        -t_{j-1}\frac{\beta_{j-1
        }}{\xi_{j-1}\omega_j} \\
        \frac 1 \omega_j\\
       \end{bmatrix}$, and $\widehat y_j=\begin{bmatrix}
        -\widehat y_{j-1}\frac{\beta_{j-1}}{\xi_{j-2}\omega_j} \\
        \beta_{j-1}e_{j-1}^Ty_{j} +\frac{\alpha_j}{\omega_j}\\
       \end{bmatrix}+\frac{\beta_{j-1}}{\omega_j} e_{j-1}    $
    }
    Compute $\eta=\widehat q^TA\widehat q$\;
    
    Set $J_{1:j,j}=\widehat y_j-\frac{\beta_j^2}{\xi^2}\frac{\xi_j - \eta}{\omega_j}t_j$ and $J_{j,1:j}=J_{1:j,j}^T$\;
  % Set $j=j+1$
  }
  \caption{$Q_m$-less rational Lanczos.\label{alg:rational_lanczos}}
\end{algorithm}

{\color{black}
The following proposition shows that the 
computation of $K_j^{-1}e_j$ and $K_j^{-T}e_j$ in Lemma~\ref{lemma:thomas} by
means of the LU factorization is backward stable. 
This result ensures that  using Gaussian elimination does 
not introduce any instability in the update of $J_j$ in
 Theorem~\ref{th:J_jlastcolumn}. 

\begin{proposition}\label{lemma:stabrec}
 Assume that the elements $\alpha_j$ and $\beta_j$ of the matrix $K_j$ are computed exactly. Moreover, let the matrix $A$ be symmetric positive (negative) definite, and the shifts $\xi_j$ be negative (positive).
 If the unit roundoff  is small enough, then the solutions of the systems $K_j y = e_j$ and $K_j^T t = e_j$ computed by the recurrences~\eqref{eqn:yt} and \eqref{eqn:diagU} are backward stable.
\end{proposition}

The proof is a direct consequence of the stability analysis
in Proposition~\ref{lemma:stabLU}. % in Section~\ref{sec:LUstab}.
}

% {\color{red}
% NON METTEREI
%  \begin{remark}
% %  We would like to mention that 
% % {\color{black} In} our extensive numerical testing
% %   we did not detect any stability issue coming from the LU-based computation of $K_j^{-1}e_j$ and $K_j^{-T}e_j$ (Lemma~\ref{lemma:thomas}) and thus in the proposed update of $J_j$ (Theorem~\ref{th:J_jlastcolumn}).
% %   Nevertheless, 
% Note that the rational Lanczos iterations also work for symmetric indefinite matrices. In this case, a safer approach may involve the QR factorization of $K_j$ and $K_j^T$. This approach could be implemented by exploiting the tridiagonal structure of these matrices and applying a single Given rotation per iteration. Formulas similar to the ones in Thereom~\ref{th:J_jlastcolumn} can also be derived in this case.
%  \end{remark}
% } 

\subsection{The $Q_m$-less procedure}\label{The overall procedure}
The implementation of the proposed memory saving method is summarized in Algorithm~\ref{alg:rational_lanczos}. 
%As already mentioned, the presented version is memory saving as it avoids the allocation of the whole basis $Q_j$.
Step 4 relies on the fact that solving a single linear system with $p$ right-hand sides is more 
efficient than sequentially solving $p$ systems with the same coefficient matrix. 
Indeed, assuming for instance that a sparse direct solver is used, the 
symbolic analysis phase {\color{black} and the factorization step  
can be performed once, for all the considered right-hand sides. The same gains might be 
obtained in the sequential solution of the two systems if a very fine tuning of the adopted linear 
solver is possible. Nevertheless, also in the latter scenario the block solution strategy is still 
advantageous thanks to a better computer handling of the dense kernels involved in the solution process. 
Moreover, the coefficient factors need to be accessed only once avoiding an increment 
in the storage requirements.
}

In Algorithm~\ref{alg:rational_lanczos} we suppose that the shifts are given. 
Alternatively, {\color{black} dynamic} shift computation strategies can be easily incorporated 
in the algorithm;
see, e.g.,  \cite{DruSim11_RKSM,Druskin2010,Guettel2013} for different shift selection strategies.
Since $A$ is symmetric, all shifts can be taken to be real.
% the shifts employed in Algorithm~\ref{alg:rational_lanczos} should 
%be all real. If this holds and $A$ and $v$ are both real, no complex arithmetic is involved in Algorithm~\ref{alg:rational_lanczos}. 

%Algorithm~\ref{alg:rational_lanczos} computes $J_m=Q_m^TAQ_m$ for a given $m$. However, 
{\color{black}  Algorithm~\ref{alg:rational_lanczos} should be equipped 
with a stopping criterion that must not involve the whole basis $Q_m$. Such a stopping criterion depends on the application of interest and
different instances are discussed in section~\ref{Applications}.} 
%As soon as this condition is satisfied -- or a maximum number of iterations is reached -- the current $J_m$ is returned along with the sought approximation. 

\begin{remark}
Algorithm~\ref{alg:rational_lanczos} can be generalized by replacing
% the block setting, in which
%corresponds to the case when 
the starting vector $v$ %is replaced by 
with a full column rank matrix 
$V\in\mathbb{R}^{n\times p}$, $p> 1$. This generates 
the {\em block} rational Krylov subspace 
${\cal K}_m(A,V,\b{\xi}_m) = {\rm range}([V, (I- \xi_1^{-1} A)^{-1} V,\ldots, \prod_{j=1}^{m-1} (I- \xi_j^{-1} A)^{-1}V])$.
%, $V\in\mathbb{R}^{n\times p}$, $p> 1$, is constructed. 
In this construction, many of the scalar quantities involved in Algorithm~\ref{alg:rational_lanczos} are
replaced by $p\times p$ matrices.
% This means that additional care for the ordering of certain operations along with the right transposition of some quantities has to be taken.
Also in this case, the matrix $J_m=Q_m^TAQ_m\in\mathbb{R}^{mp\times mp}$ can still be computed 
$Q_m$-less at low computational cost.
%storage of the full basis matrix $Q_m\in\mathbb{R}^{n\times pm}$. 
We include the corresponding implementation as Algorithm~\ref{alg:block_rational_lanczos}
in the appendix.
{Analogously to the matrix form~\eqref{eq:rat:lanc:mtx}, the recurrences in 
Algorithm~\ref{alg:block_rational_lanczos} can be written as
\begin{equation}\label{eq:block:rat:lanc:mtx}
	A Q_{m} K_m = Q_{m} H_m + \left(I- \xi_m^{-1} A\right ) \widehat{Q}_{m+1}\beta_m E_m^T ,
\end{equation}
where $\widehat{Q}_{m+1}\in\mathbb{R}^{n \times p}$, $Q_m = [\widehat{Q}_1, \dots, \widehat{Q}_m] \in \mathbb{R}^{n \times mp}$, $K_m, H_m \in \mathbb{R}^{mp \times mp}$ are block tridiagonal matrices with $(p\times p)$-size blocks, $\beta_m \in \mathbb{R}^{p \times p}$ , and $E_m = e_m\otimes I_p \in \mathbb{R}^{mp \times p}$. 
We also have the block counterpart of \eqref{eq:J:exp}, that is % relation
\begin{equation}\label{eq:block:J:exp}
J_{m} = {H}_m K_m^{-1} - W_m E_m^T K_m^{-1}, \quad  W_m = (\xi_m)^{-1} Q_m^T A \widehat{Q}_{m+1}\beta_m .
\end{equation}
%which is the block counterpart of \eqref{eq:J:exp}.

}

\end{remark}

%%%%%%%%%%%%%%%%%%%%%%%%%%%%%%%%%%%%%%%%%%%%
\section{Applications}\label{Applications}
In this section we illustrate the applicability of the $Q_m$-less rational Krylov algorithm
to a variety of problems. % where the use of the short-term recurrence can be appealing.
%thus avoiding the explicit construction of the matrix $Q_j$.
%
All numerical results %reported in this section
have been obtained by running \MATLAB{} R2017b~\cite{MATLAB} on a
standard node\footnote{CPU: 2x Intel Xeon Skylake Silver 4110 @ 2.1 GHz, 8 cores per CPU. RAM: 192 GB DDR4 ECC. {\color{black}See also {\tt https://www.mpi-magdeburg.mpg.de/cluster/mechthild}.}} of the Linux cluster \texttt{mechthild} hosted at the Max Planck Institute for Dynamics 
of Complex Technical Systems in Magdeburg, Germany.

%%%%%%%%%%%%%%%%%%%%%%%%%%%%%%%%%%%%%%%%%%%%
\subsection{Quadratic and bilinear forms}\label{sec:bilinear}
Consider a function $f$ defined on the spectrum of the symmetric matrix $A$, 
and vectors $u$, $v$. The approximation of the bilinear form 
$u^T f(A)\, v$ (or quadratic form for $u=v$)
arises in many applications including network analysis \cite{EstRod05}, %\cite{BenBoi10,EstHig10}, 
regularization problems \cite{FenReiRod16},   %\cite{MorReiSga06,FenReiRod16}, 
electronic structure calculations \cite{SaaCheSho10},
%\cite{Bar&al10,SaaCheSho10},  
solution of PDEs \cite{Lam11}, %\cite{Lam10,Lam11}, 
Gaussian processes \cite{Pettitt2002}, and many others \cite{BaiFahGol96,BenGol99,GolMeuBook10}.
 Bilinear forms can also be used to estimate the trace of $f(A)$; see section \ref{sec:trace}.

Given a large and sparse $A$, the use of (polynomial) Krylov subspace methods for the approximation of $u^T f(A)\, v$ is well established and grounded in a theoretical framework comprising orthogonal polynomials and Gauss quadrature~\cite{GolMeuBook10}.
%Specifically, the short-term recurrence characteristic  of the Lanczos method
%is valuable when dealing with huge problems in which storing capability is an important limitation. 
Under the assumptions that $u=v$ and $\|v\|=1$, the $m$-th Lanczos iteration 
produces an $m \times m$ tridiagonal matrix $T_m$ (known as \emph{Jacobi} matrix) giving the approximation
%\begin{equation}\label{eq:lan:approx}
$v^T f(A)\, v \approx e_1^T f(T_m) e_1$.
%\end{equation}
%with $e_1$ the first vector of the canonical Euclidean basis.
Such approximation relies on the so-called \emph{moment matching property}, that is
%\begin{equation}\label{eq:mmp}
$v^T A^j\, v = e_1^T (T_m)^j e_1,$ $j=0,\dots,2m-1$,
%\end{equation}
or, equivalently,
$v^T p(A)\, v = e_1^T p(T_m) e_1$,
for every polynomial $p(x)$ of degree at most $2m-1$.
Such property is connected with the Gauss quadrature approximation for a Riemann-Stieltjes integral determined by $A$ and $v$; see, e.g., \cite{GolMeuBook10,LieStrBook13}.

An analogous result can be given for rational Krylov subspaces. 
Let $Q_m$ and $J_m$ be the matrices associated with
%the columns of $Q_m$ form an orthogonal basis for 
the rational Krylov subspace $\mathcal{K}_m(A,v,\b{\xi}_m)$. % and let $J_m = Q_m^T A Q_m$.
Given $q(x) = \prod_{j=1}^{m-1}(1-x/\xi_j)$ and a polynomial $p=p(x)$ of degree at most $m-1$, it holds {\color{black}that}
\begin{equation}\label{eq:rat:exact}
   p(A)q(A)^{-1} \,v = Q_m p(J_m) q(J_m)^{-1} e_1 \in \mathcal{K}_m(A,v,\b{\xi}_m);
\end{equation}
see, e.g., \cite[Lemma 3.1]{DruKniZas09}, \cite[Lemma 4.6]{GuettelPhD.10}.
Left multiplication by $Q_mQ_m^T A$ yields
$ Q_mQ_m^T \, A\, p(A)q(A)^{-1} \,v = Q_m \, J_m \, p(J_m) q(J_m)^{-1} e_1$,
and, hence, by linearity
\begin{equation}\label{eq:rat:exact:n}
   Q_m Q_m^T p_m(A)q(A)^{-1} \,v = Q_m p_m(J_m) q(J_m)^{-1} e_1,
\end{equation}
for every polynomial $p_m$ of degree at most $m$.
The following proposition extends the moment matching property to the rational case by using ideas borrowed from Vorobjev's moment problem (see, e.g., \cite[Section~3.7.1]{LieStrBook13}). 
The result can also be obtained as an application of the results in \cite[Th.~2]{GalGriVDo96}. 
However, our approach provides a short alternative proof that, to our knowledge, 
has not yet appeared in the literature.
\vskip 0.05in
\begin{prop}\label{prop:rat:mmp}
With the previous notation for ${\cal K}_m(A,v,\b{\xi}_m)$, $Q_m$ and $J_m$,
%   Consider the matrix $J_m=Q_m^T A Q_m$, with ${\rm Range}(Q_m)=\mathcal{K}(A,v,\b{\xi}_m)$, and 
%let $q(x)$ be the polynomial  
let $q(x) = \prod_{j=1}^{m-1}(1 - x/\xi_j)$.
   Then for every polynomial $p(x)$ of degree at most $2m-1$
   $$ v^T p(A)q(A)^{-2} \,v = e_1^T p(J_m) q(J_m)^{-2} e_1. $$
\end{prop}
\begin{proof}
  For every $p_{m-1}(x)$ of degree at most $m-1$, equation~\eqref{eq:rat:exact} gives
  $$ v^T p_{m-1}(A)q(A)^{-1} \,v = e_1^T p_{m-1}(J_m) q(J_m)^{-1} e_1. $$
 Consider a polynomial $p_m$ of degree at most $m$. Given that $(Q_m Q_m^T)^2 = Q_m Q_m^T$, equation~\eqref{eq:rat:exact:n} 
becomes $Q_m Q_m^T\left( Q_m p_{m}(J_m) q(J_m)^{-1} e_1 - p_{m}(A)q(A)^{-1} \,v\right) = 0$,
  implying that the vector $Q_m p_{m}(J_m) q(J_m)^{-1} e_1 - p_{m}(A)q(A)^{-1} \,v$ is orthogonal to $\mathcal{K}_m(A,v,\b{\xi}_m)$. Therefore, for any polynomial $p_{m-1}(x)$ of degree at most $m-1$ we get
  $$ 
v^T p_{m-1}(A) q(A)^{-1} \left( Q_m p_{m}(J_m) q(J_m)^{-1} e_1 - p_{m}(A)q(A)^{-1} \,v\right) = 0, 
$$
  (where we have used the symmetry of $A$), {\color{black} that is}
$$
{\color{black}v^T p_{m-1}(A) q(A)^{-1} Q_m p_{m}(J_m) q(J_m)^{-1} e_1 = 
v^T p_{m-1}(A) p_{m}(A) (q(A)^{-1})^2 \,v.}
$$
 {\color{black}This} equality concludes the proof since {\color{black} by \eqref{eq:rat:exact} it holds that}
  $v^T p_{m-1}(A) q(A)^{-1}$ $= e_1^T p_{m-1}(J_m) q(J_m)^{-1} Q_m^T$. % by \eqref{eq:rat:exact}.
\end{proof}

 Proposition \ref{prop:rat:mmp} holds for every orthogonalization process of a rational Krylov subspace, i.e., for every orthogonal basis. As a consequence, it is not related to short recurrences.
 We also remark that similar properties have been derived in \cite[Th.~3.1]{PraRei14} for a different kind of rational Krylov subspaces. Extensions to the non-symmetric case have also been studied;
see, e.g.,  \cite{GalGriVDo96,DecBul12} among many others.
%see, e.g.,  \cite{GalGriVan94,GalGriVDo96,DecBul12,BarBenJbi16}.
%see, e.g., \cite{Wat93,GalGriVan94,GalGriVDo96,DecBul12,BarBenJbi16,Sch17}.

Thanks to Algorithm \ref{alg:rational_lanczos}, we can compute the matrix $J_m$ by means of the
$Q_m$-less short-term recurrence rational method, that is, we can compute the approximation
\begin{equation}\label{eq:bform:approx} 
 e_1^T f(J_m) e_1 \approx v^T f(A) \, v .
\end{equation}
% storing only three vectors per iteration.
 The approximation error can be characterized by adapting the results in \cite{Yse05} to our case.
 Indeed, since $A$ is symmetric we can interpret the bilinear form as a Riemann-Stieltjes integral and the approximant \eqref{eq:bform:approx} as a rational Gauss quadrature, that is,
 $$ v^T f(A) v = \int f(\lambda) \textrm{d}\mu(\lambda) \approx \sum_{j=1}^m f(\lambda_j) \theta_j = e_1^T f(J_m) e_1, $$
 where $\mu(\lambda)$ is a measure depending on the spectrum and the eigenvectors of $A$, and the $\lambda_j$'s,  $\theta_j$'s are the eigenvalues  and eigenvectors of $J_m$, respectively (see the results in \cite[Chapter 7]{GolMeuBook10} which can be easily adapted to this rational case).
 In this framework, the approximation \eqref{eq:bform:approx} is a rational quadrature rule.
 Therefore, by \cite[Eq.~(4)]{Yse05}, % we get the bound
 \begin{equation}\label{eq:bform:err:bound}
  | e_1^T f(J_m) e_1 - v^T f(A) \, v | \leq 2\|v\|^2
  \min_{\textrm{deg}(p) \leq 2m-1} \left| v^T\left( f(A) - p(A)q(A)^{-2}\right)v \right|,
 \end{equation}
where we used the notation of Proposition \ref{prop:rat:mmp} (cf. \cite[Th.~4.10]{GuettelPhD.10} related to a similar but different problem).

The block case can be treated analogously. Consider the $n \times p$ matrix $V$ and 
the $n \times n$ symmetric matrix $A$. Then using Algorithm~\ref{alg:block_rational_lanczos} 
we get the block $mp \times mp$ matrix $J_m$ and, hence, by setting
$E_1^T = [I_p, 0, \ldots, 0]\in \mathbb{R}^{p\times mp}$ we obtain the approximation
\begin{equation}\label{eq:block:biform}
    E_1^T f(J_m) E_1 \approx V^T f(A) V .
\end{equation}
%where $E_1^T = [I_p, 0, \ldots, 0]\in \mathbb{R}^{p\times mp}$.
%

Algorithm \ref{alg:rational_lanczos} can also approximate a bilinear form $u^T f(A)\, v$ with $u \neq v$. 
%(we recall that $A$ is symmetric). 
We describe various alternative strategies.
The first one is to rewrite the problem as
\begin{equation*}\label{eq:bform:approx:nonsym2}
  u^T f(A)\, v = \frac{1}{4}\left((u + v)^T f(A)\,(u + v) - (u - v)^T f(A)(u - v)\right), 
\end{equation*}
and run Algorithm \ref{alg:rational_lanczos} twice \cite[Section 7.3]{GolMeuBook10}. Such strategy maintains the same exactness of Proposition \ref{prop:rat:mmp} at twice the cost.
The second one considers the vector $u_m = Q_m^T \, u$ (computed on the fly) and the approximation
\begin{equation}\label{eq:bform:approx:nonsym}
 u_m^T f(J_m) e_1 \approx u^T f(A) v. 
\end{equation}
This approximant is exact for rational functions whose numerator has a degree up to $m-1$ and denominator $q$ from \eqref{eq:rat:exact}. 
%Note that $u_m$ can be computed on the fly by simply updating the vector, without storing the entire basis $Q_m$.
{\color{black} The third possibility 
uses \eqref{eq:block:biform} applied to the $2 \times 2$ block bilinear form
$$ 
[\; u \;\; v \;]^T \, f(A) \, [\; u \;\; v \;] = 
\begin{bmatrix}
            u^T f(A) u & u^T f(A) v \\
            v^T f(A) u & v^T f(A) v
\end{bmatrix} ,
 $$
 whose (1,2) position yields the sought-after quantity;
 see, e.g., \cite[Eqs.~(6)--(7)]{Lam11}.
%  {In this case, we obtain the approximant
%  $$ [\; u \;\; v \;]^T \, f(A) \, [\; u \;\; v \;] \approx R^T \, E_1^T f(J_m) E_1 \, R,  $$
%  where $R$ is the upper triangular factor of the QR factorization of $[u \; v]$, $E_1 = e_1 \otimes I_2$ (note that, here, $e_1 \in \mathbb{R}^m$).}
Finally, another possibility is to consider the 
rational variant of the nonsymmetric Lanczos algorithm 
in \cite{VanVanVan19}.}
%  that generates the two bases $U_m$ and $V_m$ of ${\cal K}_m(A, u, \b{\nu}_m)$
% and ${\cal K}_m(A, v, \b{\xi}_m)$, respectively, such that $V_m^TU_m = I_m$ (note that the shift sets $\b{\nu}_m, \b{\xi}_m$ are allowed to be different). Analogously to the usual nonsymmetric Lanczos iteration, its 
% rational variant also suffers from {\color{black} breakdown problems in case
% the two bases lose bi-orthogonality}.}
% We are not aware
% of a procedure implementing this general approach, although we expect that it may suffer from the
% same well known breakdown-type problems typical of the standard (polynomial) nonsymmetric Lanczos process
% \cite{GutknechtApr.1992}.} 
%{\color{magenta} We mention also the two-sided approach in \cite{Sch17}, which implements the special case $\xi_j \in \{0, \infty\}$ (extended Krylov subspace) for $A$ nonsymmetric, so that
%the basis needs to be explicitly orthogonalized.
%
% {\color{black}Thanks} to the symmetry of $A$ and in spite of a slightly higher cost,
% we favor the block approximation above because of the likely better stability properties.

\paragraph{Stopping criteria}
For a general function $f$ a cheap stopping criterion at the $m$th iteration is given by the difference between two iterates
$$ |u_m^T f(J_m) e_1 - u_{m-s}^T f(J_{m-s}) e_1|, $$
for some fixed index $s$ satisfying $1\le s < m$. 
%This criterion is computable almost for free and 
This criterion  relies on the idea that the approximation error decreases as the iterations proceed. For 
the special case of the extended Krylov subspace and Laplace–Stieltjes functions 
the convergence to $f(A)v$ is indeed monotonic \cite{Schweitzer.16}, hence the criterion is reliable.
% Indeed, given the bound \eqref{eq:bform:err:bound}, we expect to see a decreasing error behavior. 
This simple criterion can be further {\color{black} developed} following the results in \cite{Chen.Saad.18}.

A ``residual-based'' criterion can be obtained if the function $f$ is such that $y(\tau) = f(\tau A)v$ is the solution to the differential equation $y^{(d)} = Ay$, 
with $y^{(d)}$ the $d$th derivative of $y$, $d \in \mathbb{N}$, and specified initial conditions for $\tau =0$.
Indeed, let $y_m(\tau )=Q_m f(\tau  J_m) e_1$ be the approximant derived by \eqref{eq:rat:exact} and define
the differential equation residual 
\begin{equation}\label{eq:fun:res}
        r_m(\tau ) = A y_m(\tau ) - y_m^{(d)}(\tau ).
\end{equation}
Then the norm of $r_m(1)$ is commonly used as stopping criterion for Krylov subspace
approximations to $y(1)$, see, e.g., \cite{Botchev.Grimm.Hochbruck.13,Druskin1998,KnizhnermanSimoncini2010}.
Computing $\|r_m(1)\|$ would require storing $Q_m$, however, it is possible
to use an upper bound with quantities available at the current step.
Indeed, 
%let $$ b_m(t):= \beta_m \left (I-\frac 1 \xi_m A\right )q_{m+1} e_m^T K_m^{-1} f(t J_m) e_1. $$
using \eqref{eq:rat:lanc:mtx}, \eqref{eq:J:exp}, and $J_m f(\tau  J_m) e_1 = (f(\tau  J_m))^{(d)} e_1$, we get
\begin{align*}
    r_m(\tau ) =& A y_m(\tau ) - y_m^{(d)}(\tau ) = A Q_m f(\tau  J_m) e_1 - Q_m (f(\tau  J_m))^{(d)} e_1 \\
	   =&  \left(Q_{m} H_m + \beta_m (I- \xi_m^{-1} A )q_{m+1} 
e_m^T \right)K_m^{-1} f(\tau  J_m) e_1 - Q_m (f(\tau  J_m))^{(d)} e_1 \\
    %      &=  Q_{m} H_m K_m^{-1} f(\tau  J_m) e_1  + b_m(\tau ) - Q_m (f(\tau  J_m))^{(d)} e_1\\
            =&  Q_{m} J_m f(\tau  J_m) e_1 + Q_m w_m e_m^T K_m^{-1} f(\tau  J_m) e_1   \\  
	&  + \beta_m (I-  \xi_m^{-1} A)q_{m+1} 
e_m^T K_m^{-1} f(\tau  J_m) e_1 - Q_m (f(\tau  J_m))^{(d)} e_1 \\
%            &=  \left (Q_m w_m + \beta_m \left (I-\frac 1 \xi_m A\right )q_{m+1}\right) e_m^T K_m^{-1} f(t J_m) e_1 \\
	   =&  (I-Q_mQ_m^T) (I- \xi_m^{-1}A )q_{m+1}\beta_m  e_m^T K_m^{-1} f(\tau  J_m) e_1 .% = (I-Q_mQ_m^T) b_m(\tau ).
\end{align*}
% Q_m^T \beta_m(A/\xi_m)q_{m+1}
%Noting that $J_m =(H_m + \beta_m(I-A/\xi_m)q_{m+1} e_j^T )K_m^{-1}$ and 
% from the previous equality we get
%\begin{align*}
%    r_m(t) &= A y_m(t) - y_m^{(d)}(t) = A Q_m f(t J_m) e_1 - Q_m (f(t J_m))^{(d)} e_1, \\
%           &=  Q_{m} (H_m+\beta_m(I-A/\xi_m)q_{m+1} e_j^T)K_m^{-1} f(t J_m) e_1 - Q_m (f(t J_m))^{(d)} e_1 + \\ 
%           & \qquad+ (I - Q_m)b_m(t), \\
%           &= Q_{m}(J_mf(t J_m) - (f(t J_m))^{(d)}) e_1 + (I - Q_m)b_m(t), \\
%           &= (I - Q_m)b_m(t),
%\end{align*}
%where we used the fact that $J_mf(t J_m) - (f(t J_m))^{(d)}=0$. 
 Therefore 
% \begin{equation}\label{eq:res:bound}
%  \|r_m(\tau )\| \leq  % \|b_m(\tau )\|.
% 	\beta_m |e_m^T K_m^{-1} f(\tau  J_m) e_1|\,\,  \left\|(I- \xi_m^{-1} A )q_{m+1}\right\| ,
% \end{equation}
\begin{equation}\label{eq:res:bound}
 \|r_m(\tau )\| \leq  % \|b_m(\tau )\|.
	|\beta_m|\, (1+ |\xi_m^{-1}|\, \|A\| ) \, |e_m^T K_m^{-1} f(\tau  J_m) e_1|,
\end{equation}
and this holds in particular for $\tau =1$.
We recall that $t_m^T=e_m^T K_m^{-1}$ is computed iteratively during the recurrence (see Lemma~\ref{lemma:thomas}),
%, and that $A q_{m+1} = A \widehat q$ is also available during the iteration, 
hence the only extra computational cost is given by the norm and the inner product $|t_m^T f(\tau  J_m) e_1|$.
As already mentioned, the
approximant $y_m$ is exact for rational functions with a numerator of degree at most $m-1$ and denominator $q$ from \eqref{eq:rat:exact}, while \eqref{eq:bform:approx} is exact on a much larger set of rational functions; cf.~\eqref{eq:rat:exact} with Proposition \ref{prop:rat:mmp}. Therefore, the previous stopping criterion may overestimate the error of \eqref{eq:bform:approx}, 
while it is more appropriate in \eqref{eq:bform:approx:nonsym} where $v \neq u$.

{The previous stopping criteria can be extended to the block case. 
For the latter one, we can derive the bound
% $$ \left\| R_m(\tau) \right\| \leq \| \beta_m \|\,\| E_m^T K_m^{-1} f(\tau J_m) E_1 \|\, \|(I -A/\xi_m) \widehat{Q}_{m+1}\|.$$
$ \left\| R_m(\tau) \right\| \leq \| \beta_m \|\, (1 + |\xi_m^{-1}|\,\| A\|) \,\| E_m^T K_m^{-1} f(\tau J_m) E_1 \|.$
To prove the inequality above, consider the differential equation $Y^{(d)} = A Y$, with $Y(\tau) = f(\tau A) V$, and the approximant $Y_m(\tau) = Q_m f(\tau J_m) E_1$. The bound follows by using the formulas \eqref{eq:block:rat:lanc:mtx} and \eqref{eq:block:J:exp} and adapting the scalar case arguments seen above to the block case.
}

\begin{figure}[t]
    \centering
    \includegraphics[height=.55\textwidth,width=0.525\textwidth]{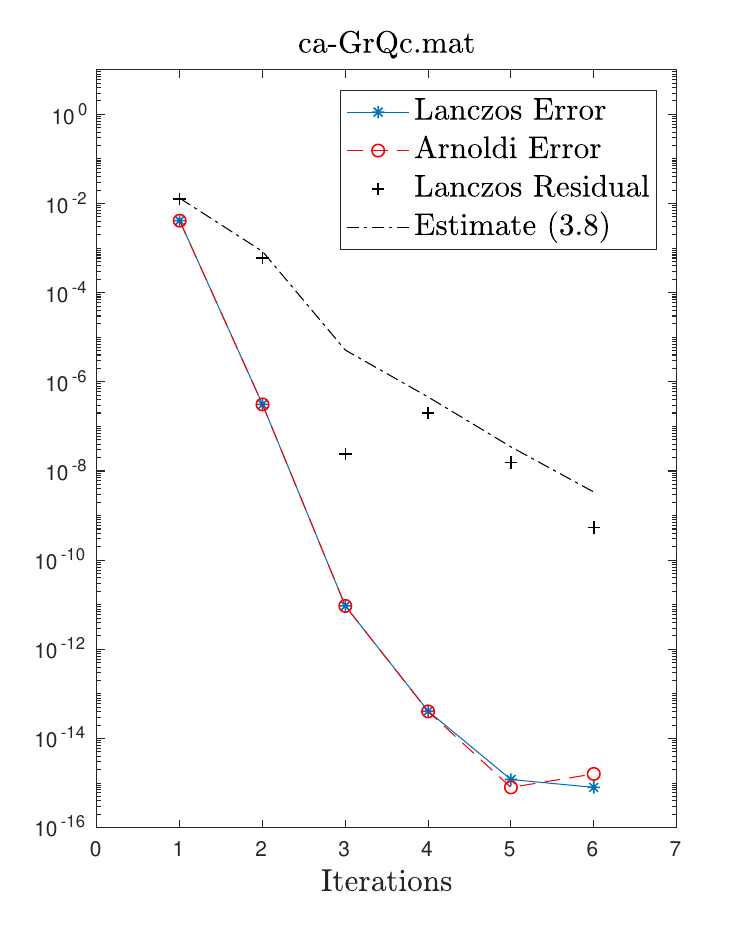}
\hskip -.30in
    \includegraphics[height=.55\textwidth,width=0.525\textwidth]{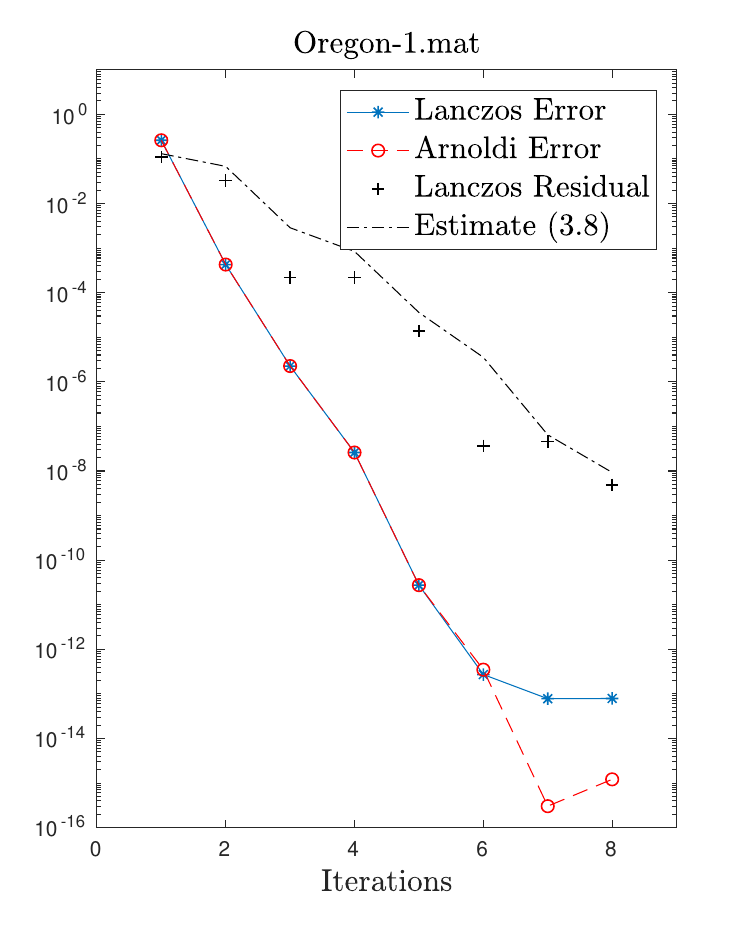}
\vskip -.2in
    \caption{Example~\ref{Ex:biform}. Convergence history of the Arnoldi error, together with the rational Lanczos error, residual norm \eqref{eq:fun:res} for $\tau=1$, and residual estimate \eqref{eq:res:bound}.
\label{fig1.biform}}
\end{figure}

  \begin{table}[t]
\centering
% \begin{tabular}{|r|c|c|c|c|}
%\hline
%  & Matrix &It. & Time (secs)   \\
%\hline
%   \multirow{2}{*}{Rat. Lanczos} & ca-GrQc & 6 & \,\,7.97e-02  \\
%   & Oregon-1 & 8 & 1.29e-01  \\
%   \hline
%   \multirow{2}{*}{Rat. Arnoldi}& ca-GrQc & 6 &  \,\,6.69e-02\\
%   & Oregon-1 & 8 & 1.30e-01  \\
%\hline
 \begin{tabular}{|r|r|r|c|c|c|}
\hline
 Matrix & size & nnz & Method & It. & Time (secs)   \\
\hline
   \multirow{2}{*}{ca-GrQc } & \multirow{2}{*}{5242} & \multirow{2}{*}{28980} & $Q_m$-less Lanczos & 6 & 0.079  \\
                             &                       &                        & Rat. Arnoldi       & 6 & 0.067  \\
   \hline
   \multirow{2}{*}{Oregon-1} & \multirow{2}{*}{11492} & \multirow{2}{*}{46818}& $Q_m$-less Lanczos & 8 & 0.129  \\
                             &                        &                       & Rat. Arnoldi       & 8 & 0.130   \\
\hline
 \end{tabular}
\caption{Example~\ref{Ex:biform} with $s=1$. Number of iterations of the 
$Q_m$-less rational Lanczos and Arnoldi methods and corresponding CPU times. Also reported
are the dimension and the number of nonzeros {\color{black} (\emph{nnz})} of 
each matrix. \label{tab_biform}}
\end{table}

\begin{example}\label{Ex:biform}
{\rm
    Let $A$ be the adjacency matrix of a network. 
For any $i\in\{1, \ldots, n\}$ the quantity $(\exp(A))_{ii}$ measures the importance of the $i$th node 
with respect to the network edge structure, the so-called \emph{$\exp$-centrality index} \cite{EstRod05}.
 %represents the \emph{$\exp$-centrality index} of the corresponding node in the network, i.e., it is a measure of the $i$th node importance with respect to the network edge structure \cite{EstRod05}. 
    We consider the symmetric normalized adjacency matrices \emph{Oregon-1} and \emph{ca-GrQc} of
size $11492$ and $5242$, respectively, from the 
SuiteSparse Matrix Collection \cite{Davis2011} and the node $i$ with the largest $\exp$-centrality. 
Both matrices are very sparse, with respectively 4 and 5 elements per row on average.
%The \emph{Oregon-1} and \emph{ca-GrQc} matrices are symmetric with sizes $11492$ and $5242$, respectively. 
We are interested in approximating the bilinear form $e_i^T \exp(A-2I) e_i$ (note that $\exp(A) = \exp(A-2I)\exp(2)$, with $A-2I$ negative definite). 
As a quality measure, we use the error between the approximation obtained by the rational Lanczos/Arnoldi methods and the quantity obtained by the \MATLAB\ function {\tt expm}.
%{\color{black} Davide: So, we dont use any of the stopping criteria we presented in the previous sections? If this is the case, then I dont understand the $s=1$ mentioned in the caption of Table~\ref{tab_biform}.} 
%The relatively small size of the two matrices has allowed us to approximate the error of the considered methods by using the outcome of the \MATLAB{} function {\verb expm } as the reference value.
%     The numerical results reported here have been obtained by running \MATLAB{} R2020a~\cite{MATLAB} on ?? {\color{red} STE: problema di uniformità degli esperimenti}.   
%    

   Figure~\ref{fig1.biform} reports the absolute errors for both approximation methods
as the iterations proceed, until the final accuracy is attained after which, not surprisingly,
 the full orthogonalization approach shows higher accuracy; {\color{black} see, e.g., Figure~\ref{fig1.biform} (right)}. See also section~\ref{sec:fpa}. 
%convergence history of the rational Arnoldi and the rational Lanczos error absolute values. 
The figure also shows the norm of the rational Lanczos residual~\eqref{eq:fun:res} for $\tau=1$ and its 
estimate~\eqref{eq:res:bound}. %The rational Arnoldi and Lanczos methods behavior are similar. 
%However, for the Oregon-1 matrix, Arnoldi achieves a better accuracy (in the ca-GrQc case, both errors converge to zero at the 6th iteration).
   In Table~\ref{tab_biform}, we report the number of iterations and corresponding CPU times, 
confirming the similar behavior of the two approaches, in terms of computational costs.
The very limited number of iterations balances the cost of the extra solves in the Lanczos process
with that of the full orthogonalization in the Arnoldi iteration.
%Also from a computational point of view the two routines perform pretty similarly pointing out that our novel implementation of the rational Lanczos iteration is indeed competitive with respect to the well-established rational Arnoldi procedure.
%    
%    
%    As expected, the rational Arnoldi is faster than the rational Lanczos, especially since both the algorithms converge quickly and the considered matrices are relatively small. Nevertheless, the rational Lanczos running time is not too far from the rational Arnoldi one.
   
}   
\end{example}

%%%%%%%%%%%%%%%%%%%%%%%%%%%%%%%
\subsection{The trace of a matrix function}\label{sec:trace}
A problem strictly related to that of approximating a quadratic form is given by the
approximation of the trace of a matrix function, tr$(f(A)) = \sum_{i=1}^n ( f(A))_{ii} = \sum_{i=1}^n f(\lambda_i)$,
where we assume that $A$ is symmetric with eigenvalues $\lambda_1, \ldots, \lambda_n$, and $f(\lambda_i)$ is
well defined
for $i=1, \ldots, n$. The approximation of the trace by means of its definition is overly expensive for large
matrices, since it requires estimates for $e_i^T f(A) e_i$ for all $i=1, \ldots, n$.

A popular strategy is the use of a Monte-Carlo approximation.
Let $Z$ be a discrete random variable with values $\{1, -1\}$ with probability 0.5, and let $z$ be
a vector of $n$ independent samples of $Z$. Then $z^T f(A) z$ is an unbiased estimator for ${\rm tr}(f(A))$;
see, e.g., \cite{Hutchinson.89}.
By exploiting this result, one can generate $\ell$ sample vectors $z^{(k)}$, $k=1, \ldots, \ell$,
estimate $(z^{(k)})^T f(A) z^{(k)} \approx \tau_k$ by means of the procedure from section~\ref{sec:bilinear}, and
obtain
\begin{equation}\label{eq:trace:stoch}
{\rm tr}(f(A)) \approx  \frac 1 {\ell} \sum_{k=1}^\ell \tau_k.
\end{equation}
For the polynomial Lanczos method and $f(\lambda)=\lambda^{-1}$, this was analyzed in \cite{BaiFahGol96}; see
also \cite{Meurant.09}. The effectiveness of the overall approach
for general functions of symmetric matrices has been established in \cite{UbaruChenSaad.17}; see also 
{\cite{Mey2021} for an improved method for stochastic trace approximation,} and
\cite{cortinovis2020randomized} and its references for general randomized approaches.

In the past few years probing methods have also emerged as an important alternative, especially
in network analysis. These differ from Monte-Carlo approximations for the selection of the ``probing''
vectors $z^{(k)}$, which are then used to estimate $(z^{(k)})^T f(A) z^{(k)}$
 by means, for instance, of Krylov methods; see, e.g., \cite{frommer2020analysis,Bentbib2021,Mey2021}.
%; 
%see, e.g., \cite{frommer2020analysis,Mey2021}, where the polynomial Lanczos algorithm is employed.

In all these strategies, a key step is the use of the Lanczos procedure to obtain
an approximation to $(z^{(k)})^T f(A) z^{(k)}$; a block method suggests itself. 
%From the results in section~\ref{sec:bilinear} we deduce that 
Rational Lanczos can effectively be used  to
speed up convergence, in terms of number of iterations, with respect to polynomial approaches. 
A disadvantage of the rational approach lies in the solution of linear systems with the 
possibly large matrix $I-\xi_j^{-1} A$, whose cost depends on the sparsity structure of $A$.
%Nonetheless, the rational approach can still be appealing for very sparse or banded matrices with
%small bandwidth, for which the computational costs of the linear system solustions are linear in the problem dimension.

  \begin{table}[t]
\centering
 \begin{tabular}{|r|r|r|r|l|c|c|}
\hline
 size & $p$ & $\delta$ & nnz & Method & It. & Time (secs)   \\
%   \multirow{2}{*}{Rat. Lanczos} & 1 000 & 6 & \,\,2.61e-02  \\
%   & 10 000 & 5 & 1.82e0  \\
%   \hline
%   \multirow{2}{*}{Rat. Arnoldi}& 1 000 & 6 &  \,\,6.34e-02\\
%   & 10 000 & 5 & 1.99e0  \\
%  & n &It. & Time (secs)   \\
\hline
   \multirow{4}{*}{1 000} & \multirow{4}{*}{20} & \multirow{2}{*}{0.02} & \multirow{2}{*}{2 260} & $Q_m$-less Lanczos & 6 & 0.026  \\
                         &                     &                       &  & Rat. Arnoldi       & 6 & 0.063  \\
  &  & \multirow{2}{*}{0.06} & \multirow{2}{*}{11 440} & $Q_m$-less Lanczos & 7 & 0.047  \\
                         &                     &                       &  & Rat. Arnoldi       & 7 & 0.094  \\                       
                         
   \hline
   \multirow{4}{*}{10 000} & \multirow{4}{*}{200} & \multirow{2}{*}{0.002} &  \multirow{2}{*}{11 278} & $Q_m$-less Lanczos & 5 & 1.82 \\
                          &                      &                         &  & Rat. Arnoldi       & 5 & 1.99  \\
&  & \multirow{2}{*}{0.006} &  \multirow{2}{*}{21 238} & $Q_m$-less Lanczos & 10 & 3.76 \\
                          &                      &                         &  & Rat. Arnoldi       & 10 & 5.96  \\
\hline
 \end{tabular}
\caption{Example~\ref{Ex:trace}. Number of iterations of the rational 
Lanczos/Arnoldi method needed to reach the maximal achievable accuracy along with 
the CPU time in seconds. The approximation space dimension is at most
$p\cdot$(\#~It.+1). {\color{black} \emph{nnz} refers to the number of nonzero elements in the considered matrices.}\label{tab_trace}}
\end{table}

\begin{figure}[htb]
    \centering
    \includegraphics[height=.55\textwidth,width=0.525\textwidth]{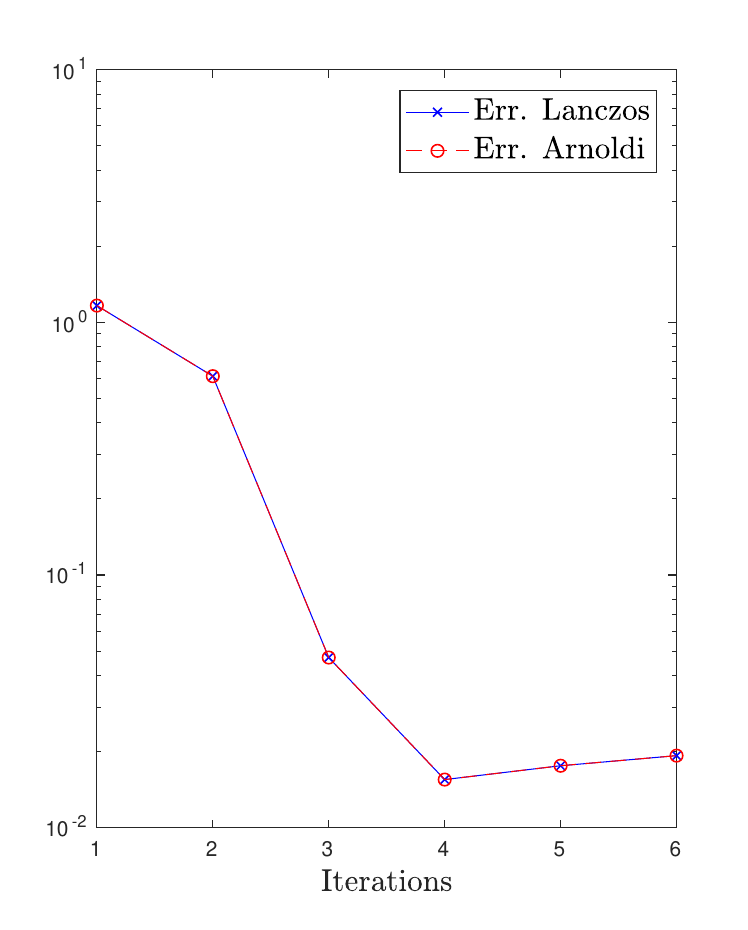}
\hskip -.30in
    \includegraphics[height=.55\textwidth,width=0.525\textwidth]{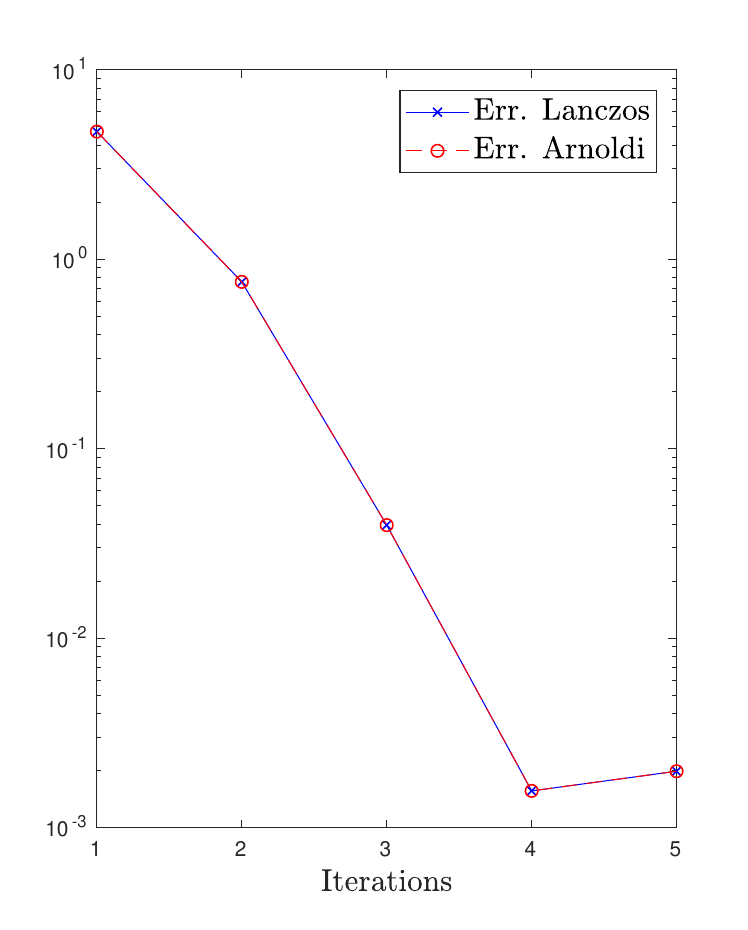}
\vskip -.2in
    \caption{Example~\ref{Ex:trace}. Convergence error history of the trace estimation \eqref{eq:trace:stoch} obtained at each step of the rational block Lanczos/Arnoldi methods. \label{fig1.trace}}
\end{figure}

 Applications related to Gaussian processes often require estimating a parameter 
$\phi$ by maximizing the so-called \emph{log-likelihood function}:
 \begin{equation}\label{eqn:gaussian}
\log(p\, |\, \phi) = \frac{1}{2}\log \det(A(\phi)) - \frac{1}{2}x^T A(\phi) x - \frac{n}{2}\log(2\pi),
\end{equation}
 where the positive definite $n \times n$ matrix $A(\phi)$ is the inverse of the covariance matrix parametrized by $\phi$, and $x$ is a given vector; see, e.g., \cite{UbaruChenSaad.17}.
 Estimating $\log \det(A(\phi))$ constitutes the main computational cost in (\ref{eqn:gaussian}).
The relation
 $\log \det(A(\phi)) $ $= {\rm tr} \log(A(\phi))$ allows one to use
 the stochastic trace estimator in \eqref{eq:trace:stoch} to reduce the 
overall computational cost.  The values $\tau_k$ can be obtained 
by using the block rational Lanczos algorithm to
approximate the bilinear form  $B^T \log(A(\phi)) B$, with $B = [z^{(1)}, \dots,z^{(p)}]$.

\begin{example} \label{Ex:trace}
 {\rm We consider the model in \cite{Pettitt2002}; see also \cite{frommer2020analysis}.
 %Similarly to what is done in \cite{frommer2020analysis}, we consider the model in \cite{Pettitt2002}. 
 We generate uniformly distributed random pairs $s_i \in [0,1]^2$, $i=1,\dots,n$, representing points on the real 
plane. A random Gaussian variable is associated with each point $s_i$. The model 
describes the association between random variables observed at fixed sites in the Euclidean space, thus
 imposing a neighborhood structure to the points. More precisely, two points $s_i, s_j$ are associated
 if and only if the Euclidean distance $d_{ij} = \| s_i - s_j\|$ is smaller than a given parameter $\delta > 0$.
 Such a structure defines a planar graph.
 We considered $\phi=20$, while the entries of $A(\phi)$ are defined as follows
 $$ [A(\phi)]_{ij} = \left\{ \begin{array}{ll}
                      1 + \phi \sum_{k=1,\, k \neq i}^{n} \gamma_{ik}, & i = j, \\ 
                      - \phi \gamma_{ij}, & i \neq j, 
                     \end{array} \right.
 $$
 with the so-called \emph{reciprocal choice} \cite{Pettitt2002}
  $$ \gamma_{ij} = \left\{ \begin{array}{ll}
                      1 - {d_{ij}}/{\delta}, & 0 < d_{ij} < \delta, \\ 
                      0, & {\rm otherwise}.
                     \end{array} \right.
 $$
Table~\ref{tab_trace} reports the results obtained with the $Q_m$-less Lanczos and rational Arnoldi algorithms for different values of $n$, $p$, and $\delta$.
%
%From the definition of $A(\phi)$, it is clear that 
The sparsity pattern of $A(\phi)$ depends on $\delta$; the larger $\delta$, the denser $A(\phi)$. 
 %Having a coefficient matrix with a sizable number of nonzero entries may lead to expensive linear system solves with a consequent impact on the computational cost of the overall rational Lanczos and Arnoldi routines. In particular, one may think that rational Lanczos could significantly suffer this aspect due to the linear system with $2q$ right-hand sides we need to solve at each iteration. 
In spite of the possible cost increase in system solves,
the rational Lanczos method turns out to be faster than rational Arnoldi for all the parameters we tested;
see Table~\ref{tab_trace}. This may be related to an increased cost of the full orthogonalization step in rational 
Arnoldi, which seems to suffer the large rank of the matrix $B$.
%, which leads to a rather expensive (full) orthogonalization step in the Arnoldi iteration. 
%The cost of the full orthogonalization wipes out the computational gains coming from the solution of a linear system with 
%only $q$ right-hand sides.}
% 
%  In our experiments, we considered two cases: $n=1\, 000, 10\, 000$. 
% {\color{red}Respectively, we set the parameters $\delta = 0.02,  0.002$ and $p=20, 200$, while $\phi = 20$ for both cases.
% VAL per expes di Davide: dovremmo considerare diversi valori di $\delta$, che influenzano la sparsita' e quindi il
% costo di Lanczos....}
Once again, we used as accuracy measure the error between the computed quantity and
the value obtained by means of the \MATLAB\ {\tt logm} function.

 For each iteration of the rational block Lanczos/Arnoldi algorithm, 
Figure~\ref{fig1.trace} displays the error in the trace approximation for
the rational block Lanczos/Arnoldi algorithms as the iterations proceed. 
% The stopping criterion was based on the 
% maximal achievable accuracy allowed by the considered stochastic strategy was reached. 
The accuracy reached in the last iterations of the examples agrees with the estimated 
achievable accuracy of the stochastic strategy we used; see, e.g., \cite{Mey2021}.
The algorithms {\color{black} behave almost identically in terms of the error}. 
% trend}.
%have the same convergence behavior. 
% {\color{red}VAL: TO BE EDITED quando ci saranno i nuovi delta. In this example, the rational Lanczos method turns out to be slightly faster than rational Arnoldi. This is due to the quite large rank of the matrix $B$, which leads to a rather expensive (full) orthogonalization step in the Arnoldi iteration.} 

%  
%  running time is slightly larger than the Arnoldi one for $n=1000$, while the rational Arnoldi running time results significantly larger than the rational Lanczos one for $n=10000$. The latter result can be explained since the Arnoldi long recurrences make it more expensive when dealing with blocks $B$ with large sizes. STE: CHE NE DITE?
}
 
\end{example}

%%%%%%%%%%%%%%%%%%%%%%%%%%%%%%%
\subsection{$\mathcal{H}_2$-norm computation}\label{H2 norm computation} 
We consider linear, time-invariant (LTI) systems of the form
\begin{equation}\label{eq:LTI}
 \Sigma:\left\{\begin{array}{l}
\dot x(t)= Ax(t)+Bu(t),\quad x(0)=x_0,\\
y(t)=Cx(t),\\
               \end{array}
\right.
\end{equation}
where $A\in\mathbb{R}^{n\times n}$ is stable, that is its spectrum is contained in the left-half open complex 
plane $\mathbb{C}_-$, and  $B\in\mathbb{R}^{n\times p}$, $C\in\mathbb{R}^{q\times n}$ are low rank, i.e., $p+q\ll n$. The $\mathcal{H}_2$-norm of $\Sigma$ is defined as follows 
$$\|\Sigma\|^2_{\mathcal{H}_2}=\text{trace}(C\mathbf{Q}C^T)=\text{trace}(B^T\mathbf{P}B),$$
where $\mathbf{Q}$, $\mathbf{P}\in\mathbb{R}^{n\times n}$ denote the controllability and the observability Gramian, respectively, i.e.,  $\mathbf{Q}$ and $\mathbf{P}$ are the solution of the following Lyapunov equations
$$A\mathbf{Q}+\mathbf{Q}A^T+BB^T=0,\quad A^T\mathbf{P}+\mathbf{P}A+C^TC=0.
$$
See, e.g., \cite[Section 5.5.1]{AntBook05}. 
The $\mathcal{H}_2$-norm gives the maximum amplitude of the system output resulting from
input signals of the LTI system (\ref{eq:LTI}) with finite energy
(see, e.g., \cite{AntBook05,Gugercin2008}), and thus its estimation
is of interest.
%\cite{Hyland1985,Wilson1970,Gugercin2008}. 
%However, the possible large size of the involved matrices makes the estimation 
%of the $\mathcal{H}_2$-norm computationally challenging.

For $A$ of large dimension $n$, model order reduction (MOR) is used to
make the dynamical system numerically tractable \cite{AntBook05}, so that the $\mathcal{H}_2$-norm 
can also be more cheaply estimated.
Given a matrix $Q_m$ whose columns span an appropriately chosen reduction space of dimension much
lower than $n$, in MOR the following smaller system is introduced, 

\begin{equation}\label{eq:reduced_model}
%{\footnotesize
 \Sigma_m:\left\{\begin{array}{ll}
 \dot{\widehat x}(t)=J_m\widehat x(t)+B_mu(t),&\widehat x(0)=Q_m^Tx_0,\; J_m=Q_m^TAQ_m,\;  \\
 \widehat y(t)=C_m\widehat x(t),&B_m=Q_m^TB,\; C_m = C Q_m.
\end{array}\right. \
%\begin{array}{l} \, 
%\end{array}
%}
\end{equation}

%%
%which for appropriately chosen small dimensional $J_m$, $B_m$ and $C_m$ is able 
This reduced system is hopefully able to reproduce the main 
features of the original large scale setting\footnote{Usually two spaces range($Q_m$), range($W_m$)
can also be considered, so that, $J_m=W_m^T A Q_m$, {\color{black} $B_m=W_m^TB$, and $C_m=CQ_m$}. For our problem one reduction space
suffices \cite{AntBook05}.}. Rational Krylov subspaces have proven to be
particularly effective for this task \cite{AntBook05,BCOW.17,GalGriVDo96,Grimme1997}.
%: the reduction strategy consists of projecting all
%system matrices onto the generated approximation space 
For $A$ symmetric, Algorithm~\ref{alg:rational_lanczos} produces the small dimensional matrix $J_m$, while
$B_m=Q_m^TB$ and $C_m=CQ_m$ can be computed incrementally during the iteration, without
explicitly storing $Q_m$.
%
%it can be shown that $u_m$ is the feedback optimal control of the LTI system obtained by projecting $\Sigma$ onto the current subspace $\mathcal{K}_m(A,C^T,\b{\xi})$, i.e.,
%The evaluation of the $\mathcal{H}_2$-norm of the LTI system (\ref{eq:LTI})
%is an important task in many model order reduction techniques; 
%see, e.g., \cite{Hyland1985,Wilson1970,Gugercin2008}.
%
After these computations,
the $Q_m$-less rational Lanczos method can be employed to 
cheaply compute $\|\Sigma_m\|_{\mathcal{H}_2}$ as an approximation to $\|\Sigma\|_{\mathcal{H}_2}$. 
In the following we approximate $\|\Sigma\|_{\mathcal{H}_2}=\sqrt{\text{trace}(B^T\mathbf{P}B)}$ and 
we thus focus on the approximation of the observability 
Gramian $\mathbf{P}$. The same procedure can be adopted to 
compute $\|\Sigma\|_{\mathcal{H}_2}=\sqrt{\text{trace}(C\mathbf{Q}C^T)}$ if the latter formulation is preferred.

The rational Krylov subspace method can also effectively be used for solving the associated
Lyapunov equation; see, e.g., \cite{Sim16}. 
%For the
%sake of completeness we report here the main steps of the method whereas we 
%refer the interested reader to \cite{DruSim11_RKSM} and the survey paper \cite{Sim16} for further details 
%on the rational Krylov subspace method and other procedures for Lyapunov equations.
%
Given the iteratively generated matrix $Q_m$ for $\mathcal{K}_m(A,C^T,\b{\xi})$, an approximation 
to $\mathbf{P}$ is sought in the form $\mathbf{P}_m=Q_mY_mQ_m^T$, where the reduced matrix $Y_m$ is obtained
by imposing an orthogonality condition on the residual matrix $R_m=A\mathbf{P}_m+\mathbf{P}_mA+C^TC$.
In terms of the Euclidean matrix inner product, this condition can be written as $Q_m^TR_mQ_m=0$.
Substituting $\mathbf{P}_m$ into the residual and using the orthogonality of the columns in $Q_m$
this yields the following
%
%In the rational Krylov subspace method we compute an approximation $\mathbf{Q}_m$ to $\mathbf{Q}$ of the form $\mathbf{Q}_m=Q_mY_mQ_m^T$ where the $mp$ columns of $Q_m$ represent an orthonormal basis of $\mathcal{K}_m(A,B,\b{\xi})$. 
%The $pm\times pm$ matrix $Y_m$ is usually computed by imposing a Galerkin condition on the residual matrix $R_m=A\mathbf{Q}_m+\mathbf{Q}_mA^T+BB^T$, i.e., we require $Q_m^TR_mQ_m=0$. It is easy to show that such a Galerkin condition is equivalent to computing $Y_m$ by solving the 
%--------------
reduced Lyapunov equation
$$J_mY_m+Y_mJ_m+E_1\b{\gamma\gamma}^TE_1^T=0,$$
where $J_m=Q_m^TAQ_m$ and $C^T=Q_1\b{\gamma}$ for a nonsigular $\b{\gamma}\in\mathbb{R}^{q\times q}$.
Hence, the small size solution $Y_m$ can be computed by means of a Schur decomposition based strategy \cite{Sim16}.
%Since $\mathbf{Q}_m=Q_mY_mQ_m^T$, at a first glance it seems that the whole basis $Q_m$ is necessary for computing an approximation to $\|\Sigma\|_{\mathcal{H}_2}$ so that Algorithm~\ref{alg:rational_lanczos} cannot be employed for this task. However, 
Using the computed quantities we can write
\begin{align*}
\|\Sigma\|^2_{\mathcal{H}_2}=& \text{trace}(B^T\mathbf{P}B)\approx 
\text{trace}(B^T\mathbf{P}_mB)=
\text{trace}\left(\left(B^TQ_m\right)Y_m\left(Q_m^TB\right)\right) =:
\|\Sigma_m\|^2_{\mathcal{H}_2}.
\end{align*}
All the required quantities can be computed without ever storing the whole matrix $Q_m$.
%%
%The matrix $B_m=Q_m^TB\in\mathbb{R}^{qm\times p}$ can be computed iteratively one column at the time as the
%space is expanded,  thus avoiding the allocation of the whole $Q_m$.
%on the fly by performing $qp$ inner products so that the allocation of the whole $Q_m$ can be avoided.

{\it Stopping criterion.} 
For the $\mathcal{H}_2$-norm computation we propose to check the distance between two 
subsequent norm approximations, that is, for $1\leq s\leq m-1$,
%\color{black}
%%
\begin{align}\label{eq:stop_H2}
 \frac{\left|\|\Sigma_m\|_{\mathcal{H}_2} - \|\Sigma_{m-s}\|_{\mathcal{H}_2}\right|}{\|\Sigma_m\|_{\mathcal{H}_2}} =&\frac{\left|\sqrt{\text{trace}\left( B_m^TY_mB_m\right)}-
\sqrt{\text{trace}\left (B_{m-s}^TY_{m-s}B_{m-s} \right)}\right|}{\sqrt{\text{trace} \left( B_m^T
Y_mB_m\right)}}.
\end{align}
%%

%Davide: A different option I found in the literature is to compute $\frac{\|\Sigma_m-\Sigma_{m-s}\|_{\mathcal{H}_2}}{\|\Sigma_m\|_{\mathcal{H}_2}}$
%where $\Sigma_m-\Sigma_{m-s}$ is the \emph{error} system given by 
%%%
%$$
%\Sigma_m-\Sigma_{m-s}:\left\{\begin{array}{l}
% \dot{\widetilde x}(t)=\begin{bmatrix}
% J_m & \\
% &  J_{m-s}\\
%         \end{bmatrix} 
%\widetilde x(t)+\begin{bmatrix}
%                 B_m\\
%                 B_{m-s}\\
%                \end{bmatrix}
%u(t),\\
%\\
% \widetilde y(t)=[C_m,-C_{m-s}]\widetilde x(t).
%\end{array}\right.$$
%%%
%If $Y_m$ and $Y_{m-s}$ are the solutions of the corresponding projected equation~\eqref{} and 
%$Y_{m,m-s}\in\mathbb{R}^{qm\times q(m-s)}$ is such that 
%%%
%$$J_mY_{m,m-s}+Y_{m,m-s}J_{m-s}-
%E_1\gamma\gamma^TE_1^T=0,$$
%%%
%then 
%%%
%\begin{align*}
%\|\Sigma_m-\Sigma_{m-s}\|^2_{\mathcal{H}_2}=&\text{trace}\left([B_m^T,B_{m-s}^T]\begin{bmatrix}
% Y_m & Y_{m,m-s}\\                                                                        
%  Y_{m,m-s}^T & Y_{m-s}\\                                                                     \end{bmatrix}\begin{bmatrix}
%  B_m\\
%  B_{m-s}\\
% \end{bmatrix}
%\right)\\
%=&\|\Sigma_m\|^2_{\mathcal{H}_2}+\|\Sigma_{m-s}\|^2_{\mathcal{H}_2}+2\text{trace}(B_{m}^TY_{m,m-s}B_{m-s}).
%\end{align*}
%
%}
%{\color{red}
%non potremmo usare (in letterature cosa fanno?)
%$$
% | 1 - \frac{\|\Sigma_{m-s}\|_{\mathcal{H}_2}}{\|\Sigma_m\|_{\mathcal{H}_2}} | ?
%$$
%}
%%
{\color{black}
The scheme presented in this section can be easily adapted to deal with certain parametric LTI systems like those studied in, 
e.g., \cite{Baur2011}, where only the matrices $B=B(\mu)$ and $C=C(\mu)$ {\em affinely} depend on a parameter $\mu$
  belonging to a given parameter set $\mathcal{D}$.} 

\begin{example}\label{Ex.1}
 {\rm
 We consider the 2D \emph{Optical Tunable Filter} dataset %~\cite{morwiki_optical} 
available in the {\sc MORwiki} repository~\cite{morWiki} (see also \cite{HohZ04}),
giving the following LTI system 
\begin{equation}\label{eq:general_LTI}
 \Sigma:\left\{\begin{array}{rcl}
E\dot x(t)&=& Ax(t)+Bu(t),\\
y(t)&=&Cx(t),\\
               \end{array}
\right.
\end{equation}
with $n=1668$, $p=1$, and $q=5$.
The mass matrix $E$ is diagonal and positive definite so that we can consider the transformed system
\begin{equation}\label{eq:transformed_LTI}
 \widetilde\Sigma:\left\{\begin{array}{l}
\dot{\widetilde x}(t)= \widetilde A\widetilde x(t)+\widetilde Bu(t),\\
y(t)=\widetilde C\widetilde x(t),\\
               \end{array}
\right. \qquad
\begin{array}{l} \widetilde A=E^{-\frac 1 2}AE^{-\frac 1 2}, \, \widetilde B=E^{-\frac 1 2}B \\
\widetilde C=CE^{-\frac 1 2},\, \widetilde x=E^{\frac 1 2}x.
\end{array}
\end{equation}
%%
%where $\widetilde A=E^{-1/2}AE^{-1/2}\in\mathbb{R}^{n\times n}$, 
%$\widetilde B=E^{-1/2}B\in\mathbb{R}^{n\times p}$, 
%$\widetilde C=CE^{-1/2}\in\mathbb{R}^{q\times n}$, $\widetilde x=E^{1/2}x$. % with $n=1668$, $p=1$, and $q=5$.
%
This transformation does not affect the $\mathcal{H}_2$-norm of the system, since  
$\|\widetilde \Sigma\|_{\mathcal{H}_2}=\|\Sigma\|_{\mathcal{H}_2}$. 
%On the other hand, the formulation encoded in $\widetilde \Sigma$ is easier to handle in our context since the tools we need to employ in the construction of the rational Krylov subspace are readily provided. 
We construct $\mathcal{K}_m(\widetilde A, \widetilde C^T,{\b \xi}_m)$ for the approximation
of $\|\widetilde\Sigma_m\|_{\mathcal{H}_2}\approx\|\widetilde \Sigma\|_{\mathcal{H}_2}$. 
The iterations are stopped as soon as the relative quantity 
in~\eqref{eq:stop_H2} for $s=1$ is smaller than $10^{-8}$. 
Table~\ref{tab1_Ex.1} collects the results for both the $Q_m$-less Lanczos and Arnoldi methods, using
the same shifts.
Thanks to the moderate dimension of the dataset, we are also able to explicitly compute 
the $\mathcal{H}_2$-norm of the full system\footnote{$\|\widetilde\Sigma\|_{\mathcal{H}_2}$ is computed by {\tt norm(sys,2)} where {\tt sys=ss}$(\widetilde A, \widetilde B,\widetilde C,0)$.} $\widetilde \Sigma$. 
Therefore, Table~\ref{tab1_Ex.1} also reports the relative error 
$|\|\widetilde\Sigma\|_{\mathcal{H}_2}-\|\widetilde\Sigma_m\|_{\mathcal{H}_2}|/\|\widetilde\Sigma\|_{\mathcal{H}_2}$.

\begin{table}[t]
\centering
 \begin{tabular}{|r|r|c|c|c|}
\hline
  &It. & $\text{dim}\left(\mathcal{K}_{m+1}(\widetilde A, \widetilde C^T,{\b \xi}_{m+1})\right)$  & Time (secs) &  Rel. Err.\\
\hline
  $Q_m$-less Lanczos & 12 & 65& 0.219 &7.13e-9\\
   Rat. Arnoldi & 12 & 65 &0.199 &7.13e-9\\
\hline
 \end{tabular}
\caption{Example~\ref{Ex.1}. Results for $Q_m$-less rational Lanczos and rational Arnoldi methods
to achieve the prescribed accuracy. Shown are {\color{black}number} of iterations, space dimension,
CPU time and relative error.\label{tab1_Ex.1}}
\end{table}

%Despite the significant difference in the construction of 
%$\mathcal{K}_m(\widetilde A, \widetilde C^T,{\b \xi}_m)$, 
The results in Table~\ref{tab1_Ex.1} show a very similar behavior for the $Q_m$-less rational Lanczos and 
rational Arnoldi methods.
% and the same number of iterations is required to achieve the prescribed accuracy. Therefore, the two schemes end up constructing the same subspace. 
The Lanczos method allows us to store only 3 basis blocks, namely 15 vectors of length $n$, 
instead of the whole basis as in the rational Arnoldi method. In terms of CPU time,
solving $2q$ linear systems per iteration in the Lanczos approach, instead of only $q$ systems in rational Arnoldi,
does not lead to a remarkable increment in the computational efforts.

%
%Also from a computational point of view the two routines are comparable and they attain a very similar running time. This means that solving $2q$ linear systems at each iteration in the rational Lanczos method, instead of only $q$ like in rational Arnoldi, does not lead to a remarkable increment in the computational efforts of the overall procedure.
%
%We conclude by pointing out that the stopping criterion~\eqref{eq:stop_H2} turns out to be a rather robust convergence check in this example since 
%a small relative error is achieved as well.
}
\end{example}

\begin{example}\label{Ex.2}
 {\rm
We consider yet another dataset from~\cite{morWiki},  the 3D \emph{Gas Sensor} example. 
%the 3D \emph{Gas Sensor} example~\cite{morwiki_gas,BecHWetal95}. 
The LTI system has the form~\eqref{eq:general_LTI} with a diagonal positive definite mass matrix $E$, hence
the transformed LTI system in~\eqref{eq:transformed_LTI} is employed.
 In this example we have $n=66917$, $p=1$, and $q=28$. The {\color{black}large} problem dimension $n$ does not 
allow for the computation of the $\mathcal{H}_2$-norm of the full system $\widetilde\Sigma$,
 hence only the approximation $\|\widetilde\Sigma_m\|_{\mathcal{H}_2}$ is computed.
Since $p$ is significantly smaller than $q$ (the number of columns of $B$ and $C^T$, resp.), we proceed by 
constructing $\mathcal{K}_m(\widetilde A,\widetilde B,{\b \xi}_m)$ and then we compute
$\|\widetilde\Sigma_m\|_{\mathcal{H}_2}^2=\text{trace}(\widetilde C\mathbf{\widetilde Q}_m\widetilde C^T)$. 
Both methods are stopped as soon as~\eqref{eq:stop_H2} for $s=1$ becomes smaller than $10^{-8}$.
%
%
% values $p$ and $q$ are rather different so that computing 
%$\|\widetilde\Sigma_m\|_{\mathcal{H}_2}^2$ by either $\text{trace}(\widetilde B^T\mathbf{\widetilde P}_m\widetilde B)$ or 
% $\text{trace}(\widetilde C\mathbf{\widetilde Q}_m\widetilde C^T)$ is not computationally equivalent. Indeed, due to the many columns of $\widetilde C^T$, the construction of the rational Krylov subspace $\mathcal{K}_m(\widetilde A,\widetilde C^T,{\b \xi}_m)$ is much more demanding than the computation of $\mathcal{K}_m(\widetilde A,\widetilde B,{\b \xi}_m)$. We thus compute $\|\widetilde\Sigma_m\|_{\mathcal{H}_2}^2=\text{trace}(\widetilde C\mathbf{\widetilde Q}_m\widetilde C^T)$
% and we compare the rational Lanczos and Arnoldi methods. Both methods are stopped as soon as~\eqref{eq:stop_H2} gets smaller than $10^{-8}$.
% 
 Table~\ref{tab1_Ex.2} collects the results. Also for this example the rational Lanczos and Arnoldi 
methods perform similarly. This means that the cost per iteration of the 
two schemes is rather similar. On the other hand, for problems where a larger number of 
iterations is required to converge, the computational cost of the Arnoldi algorithm may 
significantly increase due to the explicit full orthogonalization.
 
\begin{table}[t]
\centering
 \begin{tabular}{|r|r|c|c|}
\hline
  &It. & $\text{dim}\left(\mathcal{K}_{m+1}(\widetilde A, \widetilde B,{\b \xi}_{m+1})\right)$  & Time (secs)   \\
\hline
  $Q_m$-less Lanczos & 19 &  20& 30.5  \\
   Rat. Arnoldi & 19 & 20 &  28.8\\
\hline
 \end{tabular}
\caption{Example~\ref{Ex.2}. Results for $Q_m$-less rational Lanczos and rational Arnoldi methods
to achieve the prescribed accuracy. Shown are the {\color{black}number} of iterations, the space dimension and
the CPU time.\label{tab1_Ex.2}}
%\caption{Example~\ref{Ex.2}. Number of iterations of the rational Lanczos/Arnoldi method needed to achieve the prescribed accuracy along with the dimension of the computed subspace and the running time in seconds.}\label{tab1_Ex.2}
\end{table}
 
 }
 \end{example}

%%%%%%%%%%%%%%%%%%%%%%%%%%
\subsection{LQR feedback control}\label{LQR feedback control} We consider once
again LTI systems $\Sigma$ of the form~\eqref{eq:LTI}, and we investigate the efficient 
computation of a different quantity related to the so-called linear-quadratic regulator (LQR) problem.
Given the LTI system~\eqref{eq:LTI} with a stable $A$, this can be stated as
$$u_*=\argmin_{u}\mathcal{J}(u), \,\, {\rm with} \,\, 
\mathcal{J}(u)=\int_{0}^{\infty}y(t)^Ty(t)+u(t)^TR^{-1}u(t)dt,$$
where ${\mathcal J}$ is a quadratic cost functional and $R$ is a $p\times p$ symmetric and positive definite matrix. 
Since $A$ is stable, this $u_*$ exists and is given by 
$u_*(t)=-Kx(t)=-R^{-1}B^T\mathbf{X}x(t)$,
where $\mathbf{X}\in\mathbb{R}^{n\times n}$ is the unique positive semidefinite and
stabilizing\footnote{That is, $A-BR^{-1}B^T\mathbf{X}$ is a stable matrix.}  solution of the following Riccati equation
\cite{Mehrmann1991}
\begin{equation}\label{eq:Riccati}
 A^T\mathbf{X}+\mathbf{X}A-\mathbf{X}BR^{-1}B^T\mathbf{X}+C^TC=0.
\end{equation}
%%
%Moreover, $\mathbf{X}$ is stabilizing, namely $A-BR^{-1}B^T\mathbf{X}$ is a stable matrix; see, e.g., \cite{Mehrmann1991}.
%
Using $u_*$ the first equation in \eqref{eq:LTI} can be written in terms of a \emph{closed-loop} dynamic
\begin{equation*}%\label{eqn:ode}
\dot x(t)=(A-BR^{-1}B^T\mathbf{X})x(t), \qquad x(0)=x_0,
\end{equation*}
%
%Given the initial state $x(0)=x_0$, 
whose solution is given by
$x(t)=\text{exp}((A-BR^{-1}B^T\mathbf{X})t)x_0$ for $t\geq0$.
% if \eqref{eq:LTI} is equipped with the initial value data $x(0)=x_0$. 
Therefore, %the optimal control $u_*$ can be written as
\begin{equation}\label{eq:optimalu}
 u_*(t)=-R^{-1}B^T\mathbf{X}\cdot\text{exp}((A-BR^{-1}B^T\mathbf{X})t)x_0 .
\end{equation}
In the following we show that for $A$ symmetric an approximation to $u_*$ can be cheaply obtained by 
 %in this section we show how to efficiently compute $u_*$ in \eqref{eq:optimalu} by the 
the $Q_m$-less rational Lanczos method; see also \cite{Alla.Kalise.Simoncini.21} for related results. % in case of a symmetric $A$.

Rational Krylov subspaces have appeared to lead to competitive methods
for solving large-scale Riccati equations; see, e.g., \cite{Simoncini2016a} and references therein.
In case a projection approach is employed,
the overall scheme is very similar to the one reported in section~\ref{H2 norm computation} for Lyapunov equations.
Once again, an approximate solution is sought in the form $\mathbf{X}_m=Q_mY_mQ_m^T$, where
 $\text{Range}(Q_m)=\mathcal{K}_m(A,C^T,\b{\xi})$, while $Y_m$ is computed, for instance, by imposing an orthogonality
(Galerkin) condition on the residual matrix. 
Explicitly imposing this condition and exploiting the property $Q_m^TQ_m=I$
determines a {\it reduced} Riccati equation to be
solved in $Y_m$ (see \cite{Binietal.book.12}), that is
% equivalent to computing $Y_m$ by solving the reduced counterpart of equation \eqref{eq:Riccati}, namely $Y_m$ is such that
%%
\begin{equation}\label{eq:Riccati_projected}
J_mY_m+Y_mJ_m-Y_mB_mR^{-1}B_m^TY_m+E_1\b{\gamma\gamma}^TE_1=0,
\end{equation}
where $B_m=Q_m^TB_m$ and $\b{\gamma}\in\mathbb{R}^{q\times q}$ is such that $C^T=Q_1\b{\gamma}$. 
Since $A$ is stable and symmetric, %namely it is negative definite, 
the matrix $J_m$ is also stable so that $Y_m$ exists and it is the unique 
positive semidefinite stabilizing solution to \eqref{eq:Riccati_projected}. % Moreover, $Y_m$ is stabilizing.

Algorithm~\ref{alg:rational_lanczos} can be employed to construct the equation (\ref{eq:Riccati_projected}) for
a growing $m$, where the rows of $B_m$ are computed iteratively during the recurrence.
At the $m$-th iteration an approximation $u_m$ to $u_*$ is obtained as 
\begin{align}\label{eq:um}
u_m(t)=&R^{-1}B_m^TY_m\,\,\text{exp}((J_m-B_mR^{-1}B_m^TY_m)t)(Q_m^Tx_0)\approx u_*(t),
\end{align}
which does not require storing the whole $Q_m$, since $Q_m^Tx_0$ can also be constructed iteratively as $m$ grows.
Thanks to the stability of $J_m$, it can be shown that
the function $u
_m(t)$ defined in~\eqref{eq:um} is indeed the optimal control of the reduced model~\eqref{eq:reduced_model}, 
namely (~\cite[Corollary 3.2]{Simoncini2016a})
$$u_m=\argmin_u\mathcal{\widehat J}(u), \quad \mathcal{\widehat J}(u)=\int_0^\infty \widehat y(t)^T\widehat y(t)+u(t)^TR^{-1}u(t)dt.$$
%%
% for the LTI system
% %%
% $$
%  \widehat\Sigma:\left\{\begin{array}{l}
%  \dot{\widehat x}(t)=J_m\widehat x(t)+B_mu(t),\quad \widehat x(0)=Q_m^Tx_0,\\
%  \widehat y(t)=C_m\widehat x(t).                                                                                                                                                                \end{array}\right.
% $$
% %%
%See \cite[Corollary 3.2]{Simoncini2016a}.

{\it Stopping criterion.} 
% To design a proper stopping criterion for Algorithm~\ref{alg:rational_lanczos}, we exploit once again an optimality property of projection-based methods. Indeed, , 
% 
The $L^2-$distance between two iterates can be employed as a measure to 
assess the quality of the computed approximation $u_m(t)$, 
\begin{equation}\label{eq:feedback_error}
  \frac{\|u_{m}-u_{m-s}\|_{L^2}^2}{\|u_{m}\|_{L^2}^2}=\frac{\int_0^\infty\|u_{m}(\tau)-u_{m-s}(\tau)\|^2d\tau}{\int_0^\infty\|u_{m}(\tau)\|^2d\tau}, \quad s\in{\mathbb N},\; 0< s < m.
\end{equation}
This quantity can be cheaply approximated because it only involves  small dimensional quantities.
Furthermore, since $J_m-B_mR^{-1}B_mY_m^T$ is a stable matrix, it holds that
$\text{exp}((J_m-B_mR^{-1}B_mY_m^T)t)\rightarrow 0$ as $t\rightarrow\infty$ 
which may lead to an exponential convergence of the quadrature formula adopted to 
approximate \eqref{eq:feedback_error}; see, e.g., \cite{Boyd1987}.

\begin{example}\label{Ex.3}
 {\rm
We consider data in~\cite[Test 1]{Alla2018}. The matrix $A\in\mathbb{R}^{n\times n}$ amounts to the 5-point finite differences discretization of the 2D Laplacian operator in the unit square~$[0,1]^2$ with zero Dirichlet boundary conditions, namely $A=1/(\bar n-1)^2\cdot(T\otimes I_{\bar n}+I_{\bar n}\otimes T)$, $T=\text{tridiag}(1,-2,1)\in\mathbb{R}^{\bar n \times \bar n}$, $n=\bar n^2$. The vector $B\in\mathbb{R}^{n}$ is such that the matrix $BB^T$ corresponds to the discrete indicating function related to the square $[0.2,0.8]^2$. Similarly, $C^TC$ with  $C\in\mathbb{R}^{1\times n}$ amounts to the discrete indicating function of 
$[0.1,0.9]^2$. We select $R=1$ and $x_0=1/(\bar n-1)\cdot\mathbf{1}_{n}$ where $\mathbf{1}_{n}\in\mathbb{R}^{n}$ is the vector of all ones.

%We compare the rational Lanczos and Arnoldi methods for 
We compare the $Q_m$-less rational Lanczos and Arnoldi methods for
the computation of the approximate optimal control $u_m(t)$ in~\eqref{eq:um}. Both 
schemes are stopped as soon as the value in~\eqref{eq:feedback_error} for $s=4$ is smaller 
than $10^{-8}$.
%
%Table~\ref{tab1_Ex.3} collects the results for different values of $n$
The results in Table~\ref{tab1_Ex.3} show that the two methods perform similarly in terms of 
convergence trend and computational efficiency. 
Figure~\ref{fig1.ex3} also reports the convergence history of the two schemes for different values of $n$, illustrating
that the lack of a full orthogonalization procedure does not affect the convergence of 
the rational Lanczos method for this example. See section~\ref{sec:fpa} for a broader discussion on this topic.

\begin{table}[t]
\centering
 \begin{tabular}{|c|c|c|c|c|}
\hline
   $n$ &It. & $\text{dim}\left(\mathcal{K}_{m+1}(A, C^T,{\b \xi}_{m+1})\right)$  & Method & Time (secs)   \\
\hline
    \multirow{2}{*}{40 000}  & \multirow{2}{*}{25} & \multirow{2}{*}{26}& $Q_m$-less Lanczos& \,\, 2.7  \\
    &&& Rat. Arnoldi & \,\,2.6\\ 
   \hline
   
    \multirow{2}{*}{160 000}  & \multirow{2}{*}{29} & \multirow{2}{*}{30}& $Q_m$-less Lanczos& 11.7  \\
    &&& Rat. Arnoldi & 11.8\\ 
   \hline
 
    \multirow{2}{*}{360 000}  & \multirow{2}{*}{29} & \multirow{2}{*}{30}& $Q_m$-less Lanczos& 28.3  \\
    &&& Rat. Arnoldi & 27.3\\ 
   \hline
 \end{tabular}
\caption{Example~\ref{Ex.3}. Number of iterations of the rational Lanczos/Arnoldi method needed to achieve the prescribed accuracy along with the dimension of the computed subspace and the running time in seconds.}\label{tab1_Ex.3}
\end{table}

\begin{figure}[t]
    \centering
    \includegraphics[width=13cm]{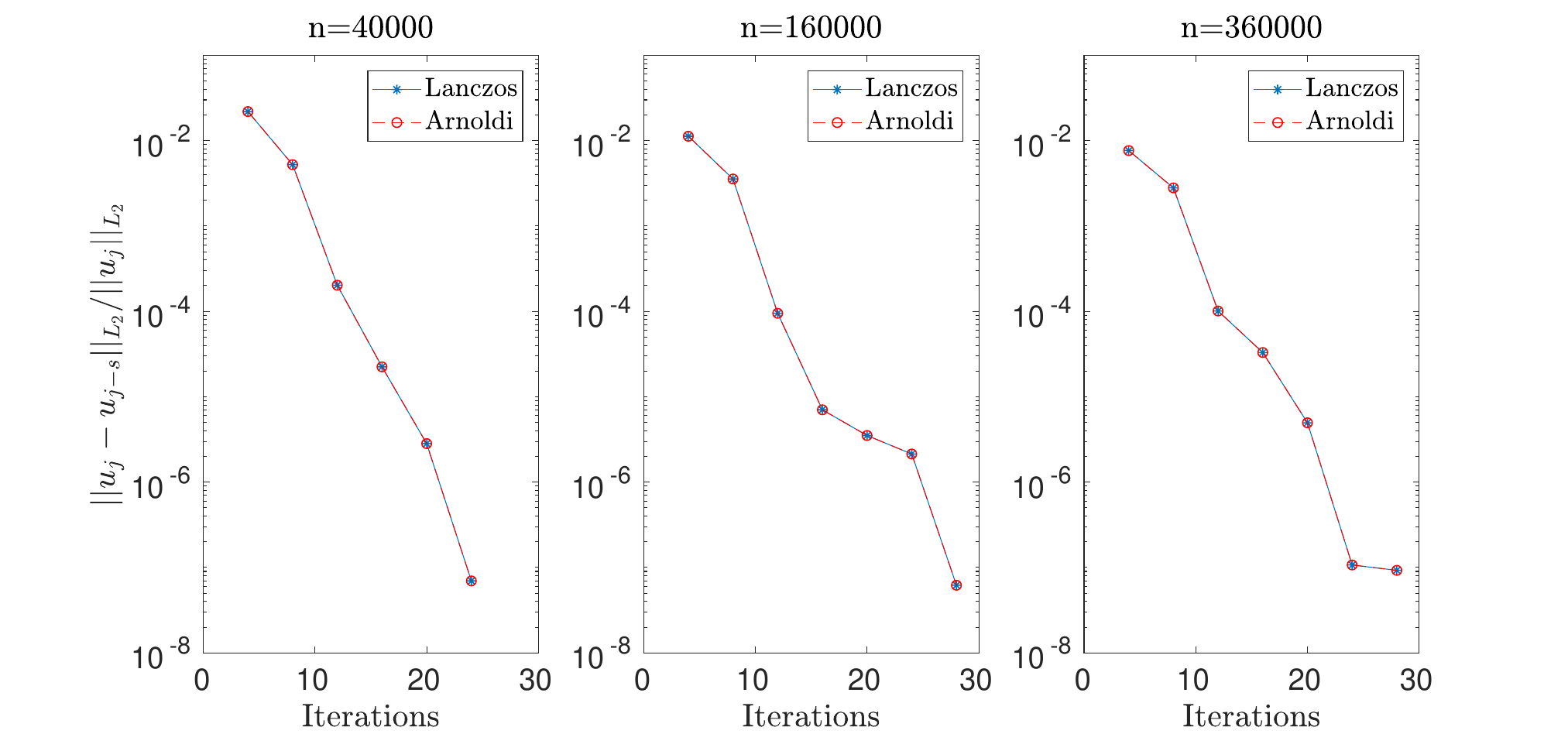}
\vskip -.10in
    \caption{Example~\ref{Ex.3}. Convergence history of the rational Lanczos and Arnoldi methods for different values of $n$ and $s=4$.}\label{fig1.ex3}
  \end{figure}

 }
 \end{example}

%\section{Numerical examples}

%%%%%%%%%%%%%%%%%%%%%%%%%%%%%%%%%%%%%%
\section{Considerations on finite precision arithmetic computations}\label{sec:fpa}
The po-\ lynomial Lanczos iteration is known to be prone to numerical instabilities, which
cause loss of orthogonality in the computed basis. This fact has been deeply investigated
by Paige in his {\color{black}seminal} PhD thesis \cite{Pai71} and successive 
works, see, e.g., \cite{Pai76}. %\cite{Pai71,Pai72,Pai76}.
Quoting \cite[p.~108]{Simon1984}, loss of orthogonality can be viewed as the result of an
amplification of each local error after its introduction into the computation,
and its growth is determined by the eigenvalue distribution of $A$ and by the starting vector $v$.

The rational Lanczos sequence is tightly related to its polynomial counterpart,
with the added difficulty of the system solves, whose finite precision arithmetic
computations may significantly increase the perturbation induced by round-off, even 
assuming a stable direct solver is employed. Notably,
to the best of our efforts, we could not find in the literature a round-off errors 
analysis for rational Krylov subpace computations
within the considered applications.
{\color{black} In this section we introduce very preliminary considerations on the 
round-off perturbation occurred in the computed quantities during the iteration.
We also report on our numerical experience, showing that the type of
loss of orthogonality in the basis seems to be similar to that analyzed in the polynomial
Lanczos method over several decades. We are conscious that {performing} a satisfactory
quantitative analysis requires sophisticated tools that go beyond our current presentation.
Hopefully, these preliminary results may be useful in a deeper analysis.}

At iteration $j$, the generation of the rational Krylov orthonormal
basis requires the following computations:
\begin{eqnarray*}
	\widetilde r&=&A q_j-\beta_{j-1}(I-\xi_{j-2}^{-1} A) q_{j-1}, \qquad
\widetilde s=(I-\xi_{j-1}^{-1} A) q_j \\
	r&=&(I- \xi_j^{-1} A)^{-1}\widetilde r, \quad s=(I- \xi_j^{-1} A)^{-1}\widetilde s \\
	q&=&r-\alpha_j s, \qquad {\rm with} \quad \alpha_j=\frac{r^T q_j}{s^T q_j}\\
    q_{j+1}&=&q/\beta_j, \quad \beta_j=\|q\|.
\end{eqnarray*}
Using standard assumptions on finite precision arithmetic computations, and assuming a
backward stable method is used for solving with the symmetric and positive definite
matrix $I- \xi_j^{-1} A$, we conjecture that an Arnoldi-type relation similar to that
in the results of \cite{Pai76} holds,  that is
$$
A Q_jK_j = Q_j H_j + (I - \xi_j^{-1} A) q_{j+1} e_j^T  +  F_j,
$$
where the matrix $F_j$ collects all round-off terms during the iteration. Here we envision
that the columns of $F_j$ will have an increasing norm as $j$ grows, according with the error accumulation
argument known for the polynomial Lanczos. We stress that $K_j$ and $H_j$ are not the
same as those computed in exact arithmetic, and that the columns of $Q_j$ are no longer 
exactly orthonormal, however we can assume that $v=Q_j e_1$.  
Although our conjecture seems to be confirmed by 
numerical experiments, a rigorous analysis leading to upper bounds for the 
elements in $F_j$ would be desirable, though it goes beyond the aim of this work.
%
%We also stress that 
{\color{black}A} different approach to the understanding of the finite precision
arithmetic behavior could also follow the {\color{black}backward error}
analysis introduced by Greenbaum for the Lanczos method 
%more related to backward error arguments 
\cite{Greenbaum.89}.
{\color{black} In particular, this approach may enlighten the interplay 
between the distribution of the eigenvalues of $A$ and the shifts $\xi_j$, 
leading to a better understanding of the rational Lanczos convergence behavior 
in finite precision arithmetic.}

\vskip 0.1in

\begin{remark}\label{rem:exactJ}
%{\rm
In finite precision it no longer holds that $J_j = Q_j^TA Q_j$, and one could even question the
	symmetry of the computed $J_j$. However, since we are mainly concerned
	with the loss of accuracy in the computation of the {\color{black}length-$n$} vectors, 
we can assume that
	$K_j^{-1}$ and $J_j:= H_j K_j^{-1} - w e^T K_j^{-1}$ are computed accurately,
	with $w$ as in the discussion after (\ref{eqn:Jj+1}), so that
$w^T=u^T \beta_j/\xi_j = \beta_j/\xi_j(\alpha_{j+1} - \eta \beta_j/\xi_j)e_j^TK_j^{-1}=: \upsilon_j e_j^TK_j^{-1}$. 
We thus have
$$
J_j = ( H_j - \upsilon_j K_j^{-T} e_j e_j^T) K_j^{-1} = H_j K_j^{-1} - \upsilon_j K_j^{-T} e_j e_j^T K_j^{-1} ,
$$
with $H_j K_j^{-1}$ symmetric. Hence, $J_j$ remains symmetric also in finite precision arithmetic,
	as long as all quantities are determined using the computed coefficients.  %$\quad\square$
%}
\end{remark}

\vskip 0.1in

According to Remark~\ref{rem:exactJ}, we thus assume 
that $J_j = H_j K_j^{-1} - \upsilon_j K_j^{-T} e_j e_j^T K_j^{-1}$  and $K_j^{-1}$
are computed exactly.
We can then write the perturbed relation as 
\begin{eqnarray*}
A Q_j &=& Q_j H_jK_j^{-1} + (I - \xi_j^{-1} A) q_{j+1} e_j^T K_j^{-1} +  F_jK_j^{-1} \\
&=& Q_j J_j + [ Q_j \upsilon_j K_j^{-T} e_j +(I-\xi_j^{-1}A) q_{j+1} ] e_j^T K_j^{-1} +  F_jK_j^{-1} \\
%&=& Q_j J_j + [ (Q_j Q_j^T - I)A\xi_j^{-1} + I] q_{j+1} \beta_j e_j^T K_j^{-1} +  F_jK_j^{-1} \\
 &=:& Q_j J_j + z_j e_j^T K_j^{-1} +  F_jK_j^{-1}.
\end{eqnarray*}
%Note that the relation $J_j = Q_j^TA Q_j$ no longer holds.
%
{Subtracting $\zeta Q_j$ for $\zeta\in\CC$ such that $A- \zeta I$ and $J_j- \zeta I$ are nonsingular, we obtain
$$
(A- \zeta I) Q_j = Q_j (J_j-\zeta I) + z_j e_j^T K_j^{-1} +  F_jK_j^{-1}.
%{\color{magenta} + G_j},
$$
Multiplying by $(A- \zeta I)^{-1}$ and $(J_j-\zeta I)^{-1}$, and rearranging the relation above we obtain} 
$$
(A- \zeta I)^{-1} Q_j = Q_j (J_j-\zeta I)^{-1} - (A- \zeta I)^{-1}(z_j e_j^T K_j^{-1} +  F_jK_j^{-1} ) (J_j-\zeta I)^{-1} .
%{\color{magenta} + G_j},
$$
%{\color{magenta}where $G_j$ collects the round-off errors of the matrix solves with $A- \zeta I$ and
%$J_j-\zeta I$. Note that this error matrix is not part of the iteration, therefore it does not
%include any amplification error, and we can assume it to be small in norm.}
We analyze the effect of this computation in the approximation of the quadratic
form $v^T f(A) v$, for which the short-term recurrence seems to 
work particularly well (see section~\ref{sec:bilinear}); see, e.g., \cite{Druskin1998} for a related 
analysis of the polynomial Lanczos method. 
Using $f(A)  = \int_\Gamma f(\zeta) (A-\zeta I)^{-1}  d\zeta$ we can write
\begin{eqnarray*}
v^T f(A) v  &=&  \int_\Gamma f(\zeta) v^T ( (A-\zeta I)^{-1} Q_j) e_1  d\zeta  \\
%v^T f(A) v  &=&  \int_\Gamma f(\zeta) v^T ( (A-\zeta I)^{-1} Q_j + G_j) e_1  d\zeta  \\
&=&  \int_\Gamma f(\zeta) v^T Q_j (J_j-\zeta)^{-1} e_1 d\zeta \\
&&  -  v^T \int_\Gamma f(\zeta) (A- \zeta I)^{-1}(z_j e_j^T + F_j) K_j^{-1} (J_j-\zeta I)^{-1}e_1  d\zeta .
%v^T f(A) (z_j e_j^T K_j^{-1} +  F_jK_j^{-1}) e_1,
\end{eqnarray*}
%{\color{magenta} where $\gamma_j = \int_\Gamma v^T G_j e_1 d\zeta$, which is of the order of machine precision.}
Therefore,
$$
v^T f(A) v = e_1^T Q_j^T Q_j f(J_j) e_1  % + \gamma_j
 -  v^T \int_\Gamma f(\zeta) (A- \zeta I)^{-1}(z_j e_j^T + F_j) K_j^{-1} (J_j-\zeta I)^{-1} e_1 d\zeta .
%+ v^T f(A) (z_j e_j^T K_j^{-1} +  F_jK_j^{-1}) e_1
$$
In exact arithmetic it would hold that $Q_j^T Q_j=I_j$ so that the quantity $e_1^T f(J_j) e_1$
would correspond to the classical approximation in the given subspace. 
In finite precision arithmetic the distance from the ideal quantity can be estimated as follows
{
\begin{eqnarray*}
|v^T f(A) v - e_1^T f(J_j) e_1| & \le &
|e_1^T f(J_j) e_1-q_1^T Q_j f(J_j) e_1| \\
 && + \left| \int_\Gamma f(\zeta) v^T (A- \zeta I)^{-1}z_j e_j^T K_j^{-1} (J_j-\zeta I)^{-1}e_1  d\zeta \; \right| \\
 && +  \left| \int_\Gamma f(\zeta) v^T (A- \zeta I)^{-1} F_j K_j^{-1} (J_j-\zeta I)^{-1}e_1  d\zeta \; \right|  \\
&=:& \mathcal I + \mathcal I\mathcal I +\mathcal I\mathcal I\mathcal I.
%&& +  |v^T f(A) z_{j}| \, |e_j^T K_j^{-1} e_1| + |v^T f(A) F_j K_j^{-1} e_1|.
\end{eqnarray*}
}
We next analyze the right-hand side terms. Let $q_1^T Q_j = e_1 + \b{\epsilon}_j$, where
we can assume that $e_1^T \b{\epsilon}_j=0$ (exact normalization) 
while the quantity $|e_k^T\b{\epsilon}_j|$ may grow with $k$. Then
\begin{eqnarray}
|e_1^T f(J_j) e_1-q_1^T Q_j f(J_j) e_1| & = &
|{\bm\epsilon}_j^T f(J_j) e_1| \nonumber\\
&=& |\sum_{\ell=1}^j ({\bm\epsilon}_j)_\ell (f(J_j) e_1)_\ell| 
\le \sum_{\ell=1}^j | ({\bm\epsilon}_j)_\ell| \, |(f(J_j) e_1)_\ell|. \label{eqn:error_bilin}
\end{eqnarray}
In exact arithmetic, it has been proved that $|e_\ell^T f(J_j) e_1|$ shows a
 decaying behavior - {\color{black}which is} possibly exponential - as $\ell$ grows, and the slope 
depends on the spectral
properties of the coefficient matrix \cite{Pozza.Simoncini.20tr}. 
If this property is maintained in finite precision arithmetic,
then each term $| ({\bm\epsilon}_j)_\ell|$ is allowed to grow as long as the
product $| ({\bm\epsilon}_j)_\ell| \, |(f(J_j) e_1)_\ell|$ remains small.

%The quantity $|\gamma_j|$ comes from the analytical derivation and we can assume it to be small.
{\color{black} 
Assuming that $K_j$ is computed exactly, 
the term $\mathcal I\mathcal I$ %$|v^T f(A) z_{j}| \, |e_j^T K_j^{-1} e_1|$ 
does not involve perturbation matrices, therefore its magnitude}
is related to the quality of the
rational Krylov space approximation in exact precision arithmetic.
Lastly, the term $\mathcal I\mathcal I\mathcal I$ %$|v^T f(A) F_j K_j^{-1} e_1| \le \|f(A)v\|\, \|F_j K_j^{-1} e_1\|$
shows that the columns of the round-off error matrix are
 {\em weighted} by the components of the vector $K_j^{-1} (J_j-\zeta I)^{-1} e_1$. 
By using the definition of $J_j$ and $K_j$, it follows that $(J_j-\zeta I)K_j = H_j - \zeta K_j - w_j e_j^T$,
which is a tridiagonal plus a rank-one matrix acting on the last column.
Setting $G_j= H_j - \zeta K_j$ with $G_j=G_j(\zeta)$ and using the Sherman-Morrison formula, we get
 $K_j^{-1} (J_j-\zeta I)^{-1} e_1 = G_j^{-1} e_1 + G_j^{-1} w_j (e_j^T G_j^{-1}e_1)/(1-e_j^T G_j^{-1}w)$.
If $G_j$ has convenient spectral properties, then the components of the vector $G_j^{-1}e_1$ have a decaying magnitude, while
the magnitude of the second term in the formula depends on $e_j^T G_j^{-1}e_1$, the last component of the first vector.
%expect that $K_j^{-1} e_1$ also has a decaying behavior, 
Hence, in this case the propagated errors
contained in the rightmost columns of $F_j$ are weighted by small values. As a consequence,
the round-off effect appears to be mitigated. Although we have experimental evidence that round-off
errors seem to only slightly affect computations, a thorough analysis is needed to make more
definitive statements.

\begin{remark}
%We mention that 
{\color{black}For $f(z)=z^{-1}$ the polynomial Lanczos analysis is 
closely related to that of the conjugate gradient (CG) method for solving linear systems.
% and a rich literature is available.  In particular,
% Loss of orthogonality also affects the superlinearity of CG. 
We refer the reader to the monograph \cite{LieStrBook13} and its references for a thorough discussion.}
\end{remark}

\begin{figure}[htbp]
 \centering
\includegraphics[height=.60\textwidth,width=0.495\textwidth]{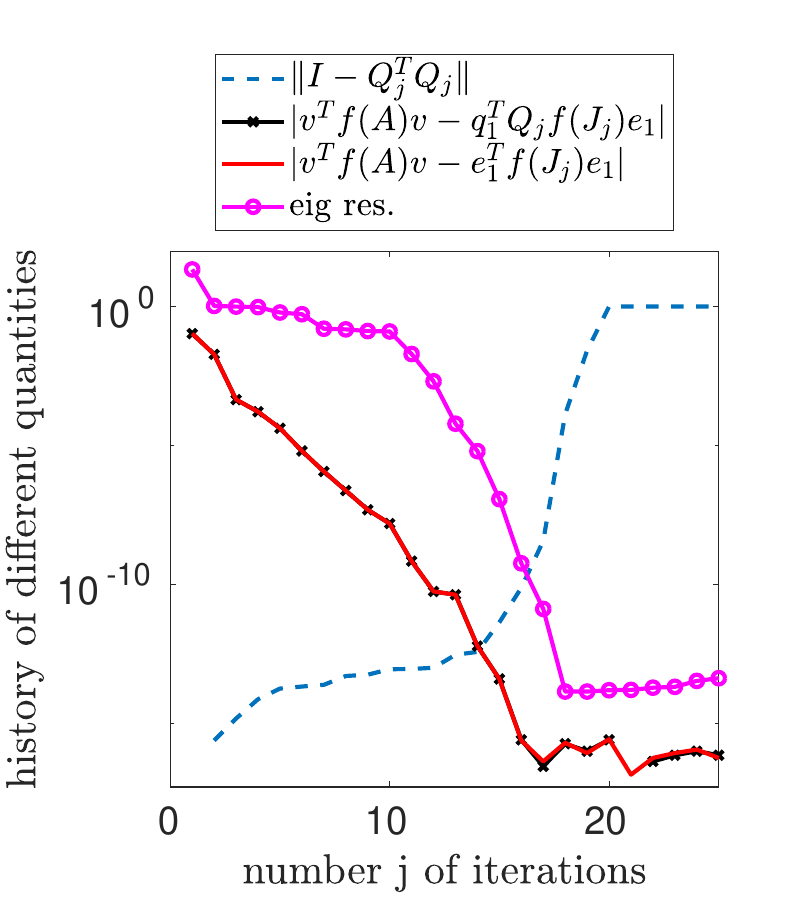}
\includegraphics[height=.60\textwidth,width=0.495\textwidth]{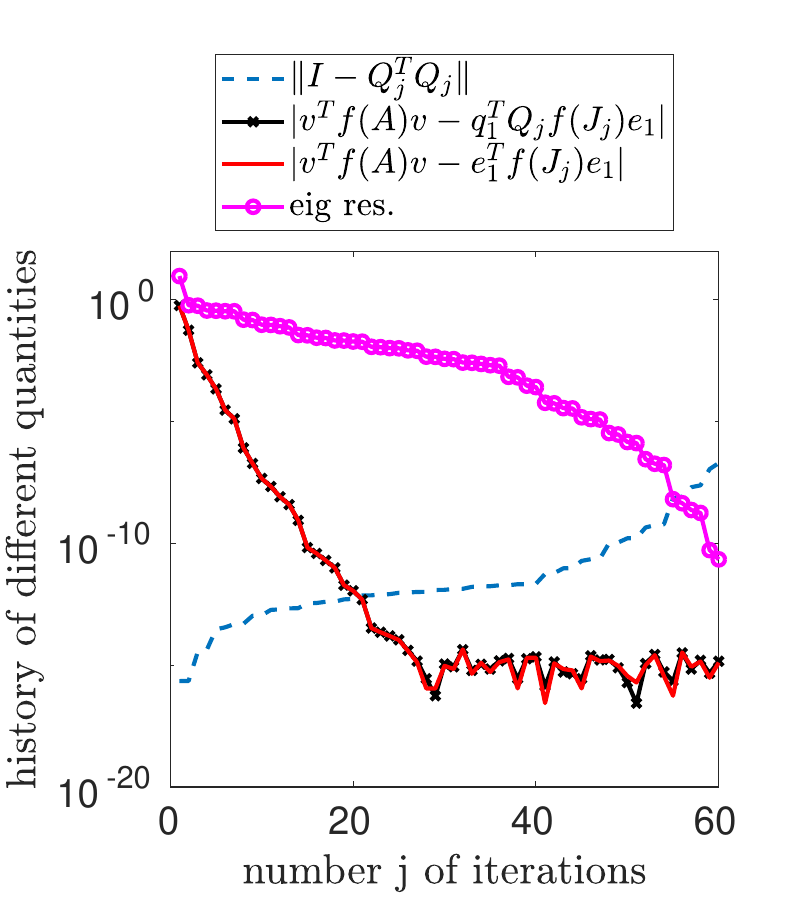}
\vskip -.10in
  \caption{ Example~\ref{ex:fpa1}. Convergence history of approximation to $v^T A^{1/2} v$.
Left: $\rho=0.45$. Right: $\rho=0.85$.  \label{fig:conv}}
\end{figure}

\begin{example}\label{ex:fpa1}
{\rm We consider an example first introduced in \cite{Strakos.91}. The matrix is diagonal with
eigenvalues equal to
$\lambda_i = \lambda_1 + (i-1)/(n-1) (\lambda_1-\lambda_n) \rho^{n-i}$, $i=1, \ldots, n$;
the parameter $\rho$ is used to control the eigenvalue distribution in the spectral
interval, so that a value of $\rho$ close to one distributes the eigenvalues almost
uniformly in the interval. We considered $n=900$, $\lambda_1=0.01$ and $\lambda_n=100$, 
	together with $f(\lambda)=\lambda^{1/2}$.
Moreover, we analyzed two values of $\rho$, that is $\rho=0.45$ and $\rho=0.85$.
The plots in Figure~\ref{fig:conv} show the error $|v^T f(A)v - e_1^Tf(J_j)e_1|$
together with the loss of orthogonality $\|I-Q_j^TQ_j\|$ and the true 
approximation error $|v^T f(A)v - q_1^TQ_jf(J_j)e_1|$ as the number $j$ of
iterations increases. Plots are reported for $\rho=0.45$ (left)
and $\rho=0.85$ (right). The eigenvalue residual norm 
$\|A x^{(j)} - x^{(j)}\lambda^{(j)}\|/|\lambda^{(j)}|$ is also shown, where
$(\lambda^{(j)}, x^{(j)})$ is the Ritz eigenpair 
with $\lambda^{(j)}$ closest to $\lambda_n$.
{\color{black}Similarly to the polynomial Lanczos method, 
for both values of $\rho$}
%we notice that, 
loss of orthogonality
is related to the convergence of the Ritz eigenpair to {\color{black}the corresponding eigenpair} of $A$. 
Concerning the bilinear form,
we first remark that the convergence to $v^T f(A)v$ is not  consistently
related to that of the Ritz eigenpair, 
and in addition convergence seems to be insensitive to the fact that $q_1^TQ_j \ne e_1$, that is,
${\bm\epsilon}_j \ne 0$.
Though convergence is slower for $\rho=0.85$ than for
	$\rho=0.45$, the last property is maintained in both cases.
}
\end{example}

In the previous example we illustrated that the accuracy obtained by $q_1^TQ_jf(J_j)e_1$
is {\color{black} similar} to that of $e_1^Tf(J_j)e_1$, and this is related
to the role of $|\sum_{\ell=1}^j ({\bm\epsilon}_j)_\ell (f(J_j) e_1)_\ell|$ in the
discussion above. The next example investigates this issue further.

\begin{figure}[htbp]
 \centering
\includegraphics[height=.6\textwidth,width=0.495\textwidth]{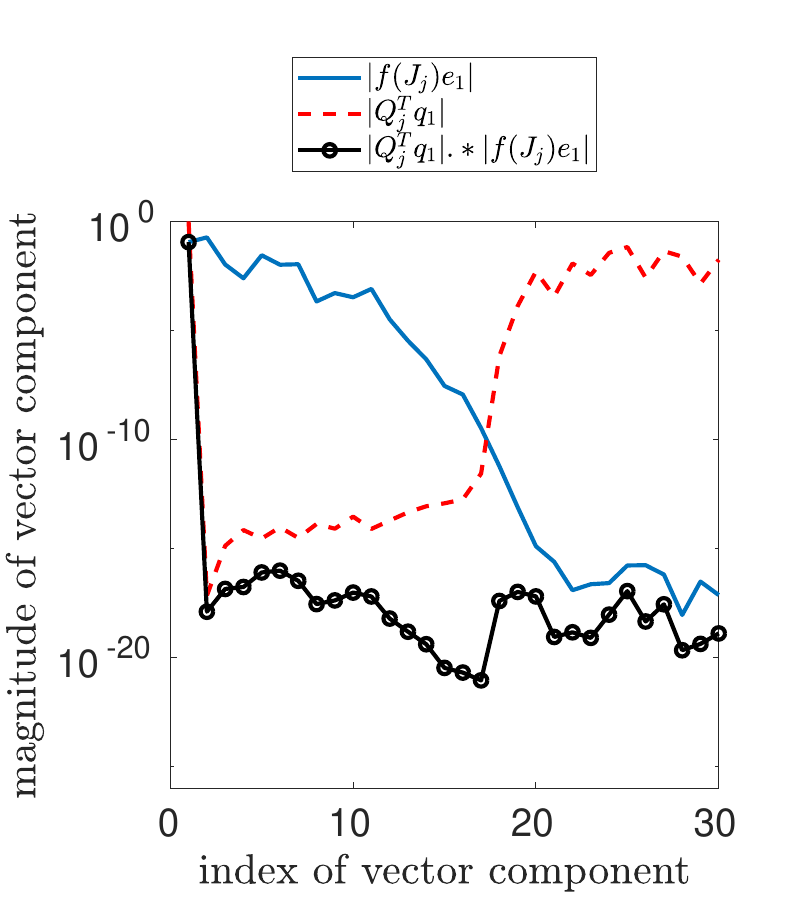}
\includegraphics[height=.6\textwidth,width=0.495\textwidth]{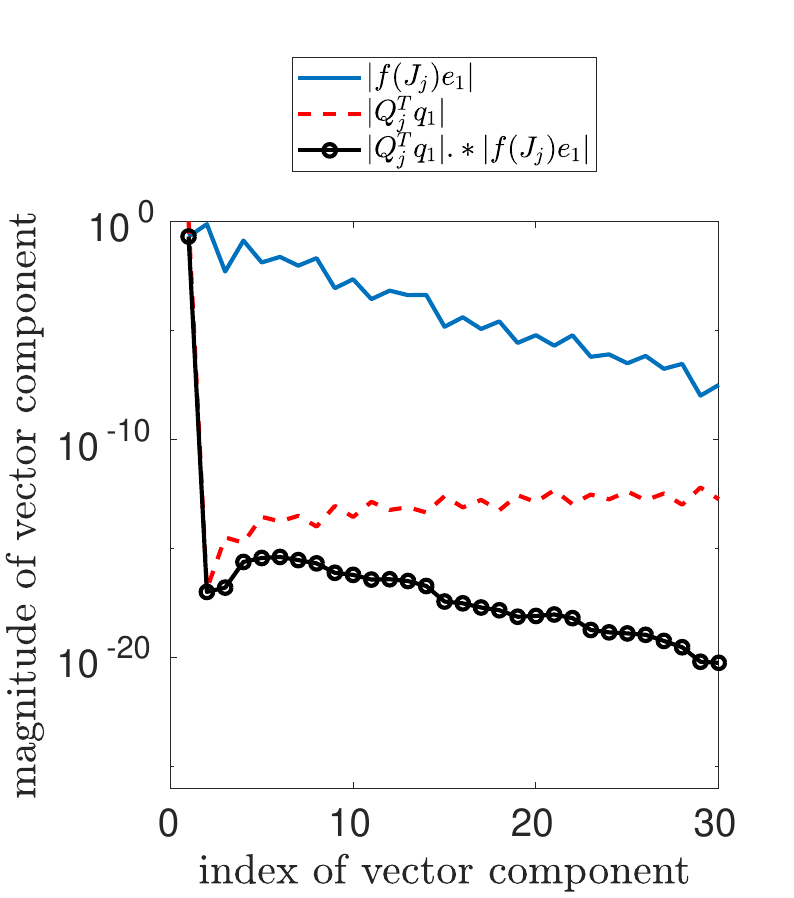}
  \caption{Example~\ref{ex:fpa2}. Magnitude of components at iteration $j=30$ in the 
	approximation to $v^T A^{1/2} v$. Left: $\rho=0.45$. Right: $\rho=0.85$. \label{fig:orth}}
\end{figure}

\begin{example}\label{ex:fpa2}
{\rm
With the same data as in {\color{black} Example~\ref{ex:fpa1}}, we focus on iteration $j=30$ and
inspect the magnitude of the {\color{black}components} of the vectors $f(J_j) e_1$ and
$q_1^TQ_j$, which give the factors in the 
sum $|\sum_{\ell=1}^j ({\bm\epsilon}_j)_\ell (f(J_j) e_1)_\ell|$.
The two plots in Figure~\ref{fig:orth} ($\rho=0.45$ on the left and $\rho=0.85$
on the right) report the quantities $|f(J_j) e_1|_\ell$, 
$|q_1^TQ_j|_\ell$ and $|q_1^TQ_j|_\ell |f(J_j) e_1|_\ell$ for
$\ell=1, \ldots, 30$.
The two figures consistently report that the increasing pattern of $|q_1^TQ_j|_\ell$ 
inversely matches the decreasing one of $|f(J_j) e_1|_\ell$, so that the
product of each component remains at the level of $10^{-15}$.
If $J_j$ were exact, the decay pattern of $|f(J_j) e_1|_\ell$ would be expected,
{\color{black}thanks} to the theoretical results reported in \cite{Pozza.Simoncini.20tr}.
{\color{black} 
The fact that round-off error does not seem to significantly alter the
interplay between $|q_1^TQ_j|_\ell$ and $|f(J_j) e_1|_\ell$ 
%{\color{black} computation} 
could be related to the theory developed by Greenbaum \cite{Greenbaum.89}.
% for the polynomial Lanczos iteration. 
For the polynomial Lanczos method,
Greenbaum showed that the computed projected matrix can be expressed as 
the output of the Lanczos iteration
applied {\it in exact arithmetic} not to $A$ but to a matrix $\tilde A$ of larger dimensions.
In particular the eigenvalues of $\tilde A$ cluster around those of $A$.
If an analogous result could be derived for the projected matrix $J_j$ computed by the
rational Lanczos  method, then the result in \cite{Pozza.Simoncini.20tr} would 
ensure that the decay pattern of $|f(J_j) e_1|_\ell$
also holds in finite precision. This would justify the persistency of the 
interplay between $|q_1^TQ_j|_\ell$ and $|f(J_j) e_1|_\ell$ in finite precision arithmetic.
We will address this intriguing issue  in future research.}

%consider the projected matrix $T_j$ computed with finite precision by $j$ iterations of polynomial Lanczos applied on a matrix $A$. Greenbaum showed that $T_j$ can be expressed as the output of Lanczos on a matrix $\tilde A$ in exact arithmetic. While the matrix $\tilde A$ has a larger size than $A$, its eigenvalues are close to $A$'s ones. If an analogous result holds for the projected matrix $J_j$ computed by rational Lanczos, then the decay pattern of $|f(J_j) e_1|_\ell$ in finite precision arithmetic is preserved.}

% {\color{black} where a perturbation model involving an augmented
% projected matrix is proposed. {\color{red}NON CHIARO, c'e' troppo poco
% dettaglio per capire. Nella frase prima c'e' scritto che la matrice proiettata
% e' augmented. Within this model, the matrix $J_j$ and thus
% the decay pattern of $|f(J_j) e_1|_\ell$ would be preserved}.}
%the practical effect of round-off error on the eigenvalue approximation,
%which involves the projected matrix and a certain augmentation of it.
}
% {\color{magenta} If an analogous model can be derived, then the decay pattern of $|f(J_j) e_1|_\ell$ in finite precision arithmetic is preserved.}

\end{example}

%%%%%%%%%%%%%%%%%%%%%%%%%%%%%%%%%%%%%%%%%%%%%%
{\color{black} %SP
\subsection{\color{black}Stability of the LU factorization of $K_j$}\label{sec:LUstab}
{\color{black}In section~\ref{On the computation of J} it was stated
that the LU factorization  $K_j= L_jU_j$ with no pivoting exists.
In this section we analyze its stability properties.}
%
% The LU factorization of the tridiagonal matrix
% $$ T = \begin{bmatrix}
%         d_1 & b_1 &     & &  \\
%         c_2 & d_2 & b_2 & & \\
%             & \ddots & \ddots & \ddots & \\
%             & & \ddots & \ddots & b_{m-1} \\
%             & & & c_m & d_m
%        \end{bmatrix}
% $$
% obtained by Gaussian elimination are given as
% \begin{align}\label{eq:LUtridrec}
% \omega_j = 1; \;     \left.
%      \begin{array}{ll} \ell_i &=\, c_i/\omega_{i-1}  \\
%    \omega_i &=\, d_i -\ell_i b_{i-1}
%    \end{array} \right \} \; i =2,\dots, m;
% \end{align}
% see, e.g.,\cite[Eq.~(9.19)]{HigBook02}.
% 
% The recurrences \eqref{eqn:diagU} are obtained by applying the recurrences \eqref{eq:LUtridrec} to the matrix $K_j$. Moreover, 
% 
%Lemma \ref{lemma:thomas} expresses the solutions of the linear systems 
%$K_j y = e_j$ and $K_j^T t = e_j$ with the recurrences \eqref{eqn:yt} and \eqref{eqn:diagU}. 
%Exploiting the LU factorization $K_j = L_j U_j$ without pivoting,
%These recurrences are obtained by solving the linear systems
%$$ U_j y =e_j, \quad L_j^T t = e_j, $$
%by backward substitution (see the proof of Lemma \ref{lemma:thomas}).
Assuming that $K_j$ is computed exactly, 
{\color{black} the stability of the LU factorization will
ensure that the recurrences \eqref{eqn:yt} and \eqref{eqn:diagU} are backward stable.
To this end, we write
$$
K_j = 
\begin{bmatrix}
1 & 0 \\
\frac{\beta_1}{\xi_1} & \widetilde K_j 
\end{bmatrix} 
\quad {\rm with}\,\,
\widetilde K_j =
\widetilde D_{j-1}^{-1} \widetilde{T}_j,
\quad
\widetilde{T}_j := 
\begin{bmatrix}
 \xi_1\hspace{-2pt}+\hspace{-2pt}{\alpha_2} & {\beta_2}    \\
  {\beta_2} &\xi_2\hspace{-2pt}+\hspace{-2pt}{\alpha_3}  &  \hspace{+2pt}\ddots & \\
        & \hspace{-2pt}\ddots  & \ddots  & {\beta_{j-1}} \vspace{4pt} \\
    &       & {\beta_{j-1}}  & \xi_{j-1}\hspace{-2pt}+\hspace{-2pt}{\alpha_j} \vspace{4pt}
\end{bmatrix} ,
$$
where $\widetilde{T}_j\in\RR^{(j-1)\times (j-1)}$
and $\widetilde{D}_{j-1} := \text{diag}(\xi_1,\dots,\xi_{j-1})$. 
}

% if the Gaussian elimination for the systems $K_j y = e_j$ and $K_j^T t = e_j$ is backward 
% stable; see \cite[Section 9.6]{HigBook02}. 

{%\color{red}
\begin{lemma}\label{lemma:K}
Let $j_*\le n$ be the first index such that $\beta_{j_*}=0$, giving subspace invariance in
(\ref{eq:rat:lanc:mtx}).
Using the notation of Section \ref{On the computation of J}, {\color{black}let $A$ be symmetric and
positive definite,  and the shifts $\xi_1,\ldots, \xi_{j_*-1}$ all be negative. 
Then the symmetric matrices $H_j$ and $\widetilde{T}_j$ are both positive definite, for $j\le j_*$.
Vice-versa, if $A$ is symmetric and negative definite with positive shifts, then 
$H_j$ is negative definite while $\widetilde{T}_j$ is positive definite.}
\end{lemma}
}

\begin{proof}
{\color{black}
We prove the result for $A$ negative definite; the positive definite case follows similarly.
We have
%Without loss of generality, assume that we can run the algorithm till the last iteration $j=n$. 
%Since 
$J_{j_*} = H_{j_*} K_{j_*}^{-1}$ so that 
$ J_{j_*} K_{j_*} = J_{j_*} + J_{j_*} D_{{j_*}-1}^{-1} H_{j_*} = H_{j_*}$,
from which 
$$ 
H_{j_*} = (I_{j_*} - J_{j_*} D_{{j_*}-1}^{-1})^{-1} J_{j_*} = (J_{j_*}^{-1} - D_{{j_*}-1}^{-1} )^{-1} .
$$
Since $J_{j_*} = Q_{j_*}^T A Q_{j_*}$ is symmetric and negative definite, and the shifts are all positive,
it follows that $H_{j_*}$ is a negative definite matrix. 
Moreover, the eigenvalues of $K_{j_*}$ are positive since $K_{j_*}=H_{j_*} J_{j_*}^{-1}$ is the product of two symmetric 
negative definite matrices. % ($K_{j_*} = H_{j_*} J_{j_*}^{-1}$). 

For any $j\le {j_*}$, %we observe that
the eigenvalues of $H_j$ are contained in the spectral interval of $H_{j_*}$, so that
$H_j$ is positive definite.  % for all $j\le {j_*}$.
%
%The matrix $H_n$ is a symmetric tridiagonal matrix with positive off-diagonal elements, i.e., a \emph{Jacobi matrix}. By the interlacing property (e.g., Theorem 3.3 \cite{GolMeuBook10}), it is easy to see that the eigenvalues of $H_j$ are contained in the spectral interval of $H_n$. As a consequence, $H_j$ is symmetric negative definite for $j=1,\dots,n$.
%
To derive the positive definiteness of $\widetilde{T}_j$, we observe that 
for any $j\le {j_*}$,
the spectrum of $K_j$ is composed of $1$ and all the eigenvalues of the submatrix
$\widetilde{K}_j =
%\begin{bmatrix}
% 1\hspace{-2pt}+\hspace{-2pt}\frac{\alpha_2}{\xi_1} & \frac{\beta_2}{\xi_1}    \\
%  \frac{\beta_2}{\xi_2}  &1\hspace{-2pt}+\hspace{-2pt}\frac{\alpha_3}{\xi_2}  &  \hspace{+2pt}\ddots & \\
%	& \hspace{-2pt}\ddots  & \ddots  & \frac{\beta_{j-1}}{\xi_{j-2}} \vspace{4pt} \\
%    &       & \frac{\beta_{j-1}}{\xi_{j-1}}  & 1\hspace{-2pt}+\hspace{-2pt}\frac{\alpha_j}{\xi_{j-1}} \vspace{4pt} 
%\end{bmatrix} = 
\widetilde{D}_{j-1}^{-1} \widetilde{T}_j$, with $j\le {j_*}$.
For $j=j_*$, the positivity of the eigenvalues of $K_{j_*}$ ensures that of the eigenvalues of $\widetilde{K}_{j_*}$.
In particular, since $\widetilde{D}_{j_*-1}>0$, this implies that $\widetilde{T}_{j_*}$ is 
positive definite. Hence, all principal $j\times j$ matrices of $\widetilde{T}_{j_*}$ are also positive definite,
with $j\le {j_*}$.}
\end{proof}
%This in turn implies that $\widetilde{K}_j$ has positive eigenvalues.

%Let $\widetilde{D}_{j-1} := \text{diag}(\xi_1,\dots,\xi_{j-1})$, and
%note that $\widetilde{D}_{j-1} \widetilde{K}_j$, is again a symmetric tridiagonal matrix with positive 
%off-diagonal elements. 
%Since ${K}_n$ has a positive spectrum, also $\widetilde{K}_n$ spectrum is positive. By the interlacing property, $\widetilde{D}_{j-1}\widetilde{K}_j$ is symmetric positive definite, and therefore $K_j$ has positive eigenvalues for $j=1,\dots,n$.

{\color{black}We can prove the backward stability of the Gaussian elimination 
procedure associated with $K_j$,
thus proving Proposition \ref{lemma:stabrec}.
Here $|M|$ is the matrix obtained by taking the element-wise absolute values of the matrix $M$.}

\begin{proposition}\label{lemma:stabLU}
 Under the assumptions of Lemma \ref{lemma:K}, if the unit roundoff {\tt u} is small enough, 
then the Gaussian elimination for the system $K_j y = e_j$ succeeds, and the 
computed solution $\hat y$ satisfies
 $$
 (K_j + \Delta K_j) \hat y = e_j, \quad |\Delta K_j| < h({\tt u}) |K_j|, \quad 
h({\tt u}) = \frac{4{\tt u} + 3{\tt u}^2+{\tt u}^3}{1-{\tt u}}. 
$$
 %with the convention that the absolute value acts elementwise on matrices.
The same holds for the system $K_j^T t = e_j$.
\end{proposition}

\begin{proof}

Following the argument in \cite[section 9.6]{HigBook02}, it is sufficient 
to prove that the LU factorization $K_j = L_j U_j$ satisfies $|L_j||U_j| = |L_j U_j|$ with 
$U_j$ having positive diagonal elements, and the result will follow
from \cite[Theorem 9.14]{HigBook02} and its proof.
{\color{black}
We restrict our attention to the matrix $\widetilde K_j$,
as the first row and column of $K_j$ are already in the desired form.

For $A$ negative definite, from Lemma~\ref{lemma:K} it follows that the 
matrix $\widetilde{T}_j=\widetilde{D}_{j-1} \widetilde{K}_j$ 
is positive definite. % by Lemma~\ref{lemma:K}. 
In particular, Theorem 9.12 in \cite{HigBook02} ensures that
the LU factorization $\widetilde{T}_j = \widetilde{L}_j \widetilde{U}_j$ satisfies the 
condition $|\widetilde{L}_j||\widetilde{U}_j| = |\widetilde{L}_j \widetilde{U}_j|$.
Therefore, the matrix $\widetilde{K}_j$ can be factorized as 
$\widetilde{K}_j = \widetilde{D}_{j-1}^{-1} \widetilde{L}_j \widetilde{U}_j$. 
Note that the matrix $\widehat L_j := \widetilde{D}_{j-1}^{-1} \widetilde{L}_j \widetilde D_{j-1}$ is lower 
bidiagonal with all the diagonal entries equal to $1$. Then
 $$ 
\widetilde{K}_j = \widehat L_j \widetilde{D}_{j-1}^{-1} \widetilde U_j = \widehat L_j \widehat U_j, 
$$
 is the unique LU factorization of $\widetilde{K}_j$, with $\widehat U_j := \widetilde{D}_{j-1}^{-1} \widetilde U_j$. 
Note that the diagonal elements of $\widehat U_j$ are positive.
Since $\widetilde{D}_{j-1}$ has positive diagonal entries, we get the following equalities
 \begin{align*}
    |\widehat L_j || \widehat U_j| & = |\widehat L_j| |\widetilde{D}_{j-1}^{-1} \widetilde U_j| = |\widehat L_j \widetilde{D}_{j-1}^{-1} ||\widetilde U_j| 
     = |\widetilde{D}_{j-1}^{-1} \widetilde{L}_j | |\widetilde U_j| = |\widetilde{D}_{j-1}^{-1} | | \widetilde{L}_j | |\widetilde U_j| \\
     & = |\widetilde{D}_{j-1}^{-1}| | \widetilde{L}_j \widetilde U_j| = |\widetilde{D}_{j-1}^{-1} \widetilde{L}_j \widetilde U_j| = |\widehat L_j \widehat U_j|.
 \end{align*}
Returning to $K_j$, and
 using the notation of the proof of Lemma \ref{lemma:thomas}, we observe that the first two computed 
coefficients in the factorization $K_j = L_j U_j$ are $\omega_1 = 1$ and $\ell_2 = \beta_1 / \xi_1 > 0$. 
Hence, it holds that $|L_j||U_j| = |L_j U_j|$ with $U_j$  having positive diagonal elements, 
concluding the proof. 

The case in which $A$ is positive definite can be proved analogously.
}
\end{proof}
}

\section{Conclusions}\label{Conclusion}
We have described a computationally and memory efficient implementation of the
symmetric rational Lanczos method. The algorithm does not require storing the
whole orthonormal basis $Q_m$ to proceed with the iterations. We have illustrated 
a number of application problems where the proposed $Q_m$-less algorithm can 
effectively be employed. Very preliminary considerations 
of finite precision arithmetic computations seem to indicate
that the behavior of the short-term recurrence rational method in this context
is similar to that 
of its polynomial counterpart, although a comprehensive analysis is required to make
more definitive statements.

\section*{Acknowledgements}
{The authors would like to thank Miroslav S. Prani{\'{c}} for an insightful discussion on
\cite{PraRei14}, and Niel Van Buggenhout for the helpful comments about the rational Krylov moment matching property.}
{\color{black} The authors are also grateful to the two anonymous reviewers for their
careful reading and for comments that led us to include section~\ref{sec:LUstab}.}

The authors are members of Indam-GNCS, which support is gratefully acknowledged. This work has also been supported by Charles University Research programs No. PRIMUS/21/SCI/009 and No. UNCE/SCI/023.

The datasets and algorithms generated during and/or analysed during the current study are available from the corresponding author on reasonable request.

\section*{Appendix}
In this section we present the block
variant of Algorithm~\ref{alg:rational_lanczos}.

\begin{algorithm}
{\tiny
\setcounter{AlgoLine}{0}
  \DontPrintSemicolon
  \SetKwInOut{Input}{input}\SetKwInOut{Output}{output}
  %%%%%%%%%%% INPUT %%%%%%%%%%%
  \Input{$A\in\mathbb{R}^{n\times n}$, $V\in\mathbb{R}^{n\times p}$, $\b{\xi}$, number of iterations $m>0$.}
  %%%%%%%%%%% OUTPUT %%%%%%%%%%%
  \Output{$J_{m}\in\mathbb{R}^{pm\times pm}$, $J_{m}=Q_m^TAQ_m$, $\text{Range}(Q_m)={\cal K}_m(A,V,\b{\xi}_m) $.}
  %%%%%%%%%%%%%%%%%%%%%%%%%%%%%%%%%%% 
  \BlankLine
  Compute a skinny QR factorization of $V$, $\widehat Q \widehat R=V$\;
  
  \While{$j\leq m$}{
  \If{$j=1$}{
  Set $\widetilde R=A\widehat Q$ and $\widetilde S=\widehat Q$\;
  }\ElseIf{$j=2$}{
  Set $\widetilde R=A\widehat Q-\bar Q\beta_{j-1}^T$ and $\widetilde S=(I-A/\xi_{j-1})\widehat Q$}
\Else{
Set $\widetilde R=A\widehat Q-(I-A/\xi_{j-2})\bar Q\beta_{j-1}^T$ and $\widetilde S=(I-A/\xi_{j-1})\widehat Q$}
    Solve  $(I-A/\xi_j)[R,S]=[\widetilde R,\widetilde S]$\;
  
  Compute $\alpha_j=\left(\widehat Q^TS\right)^{-1}\left( \widehat Q^TR\right)$\;
  
  Set $Q=R- S\alpha_j$\;
  
  Set $\bar Q=\widehat Q$\;
  
  Compute a skinny QR factorization of $Q$, $\widehat Q\beta_j=Q$\;
  
  \If{j=1}{
    Set $u_j=y_1=t_1=I_p$ and $\widehat y_1=\alpha_1$
    }\Else{
    Set $u_j=\alpha_j/\xi_{j-1}+I_p-\beta_{j-1}\omega_{j-1}^{-1}\beta_{j-1}^T/(\xi_{j-1}\xi_{j-2})$ \;
    
    Set
    $y_j=\begin{bmatrix}
        -y_{j-1}\beta_{j-1}^T\omega_j^{-1}
        /\xi_{j-2} \\
        \omega_j^{-1}\\
       \end{bmatrix}
    $,
    $t_j= \begin{bmatrix}
        -t_{j-1}\beta_{j-1
        }^T\omega_j^{-1}/\xi_{j-1} \\
        \omega_j^{-1}\\
       \end{bmatrix}$, and $\widehat y_j=\begin{bmatrix}
        -\widehat y_{j-1}\beta_{j-1}^T\omega_j^{-1}/       \xi_{j-2} \\
        \beta_{j-1}^TE_{j-1}^Ty_{j} +\alpha_j\omega_j^{-1}\\
       \end{bmatrix}+E_{j-1}\beta_{j-1}u_j^{-1}    $
    }
    Compute $\eta=\widehat Q^TA\widehat Q$\;
    
    Set $J_jE_j=\widehat y_j-t_j\beta_{j}^T(I_p -  \eta/\xi_j)\beta_{j}/\xi_ju_j^{-1}$ and $E_j^TJ_j=(J_jE_j)^T$\;
    
    Set $j=j+1$
  }
  \caption{Block rational Lanczos.\label{alg:block_rational_lanczos}}
}
\end{algorithm}

\bibliographystyle{siam}
\bibliography{mybib}

\end{document}